\documentclass[oneside,11pt,reqno]{amsart}

\usepackage[a4paper,margin=29mm]{geometry}

\usepackage{verbatim,upref,amsxtra,amssymb,amscd,graphicx}

\usepackage{mathtools}

\usepackage{color}

\usepackage{hyperref} 

\usepackage{varioref} 

\usepackage{pgf} \usepackage{tikz} \usetikzlibrary{arrows,automata} \usepackage{tkz-graph}
\usepackage{caption}

\usepackage{url}

\DeclareMathAlphabet\mathscr{U}{eus}{m}{n}
\SetMathAlphabet\mathscr{bold}{U}{eus}{b}{n}
\DeclareMathAlphabet\matheur{U}{eur}{m}{n}
\SetMathAlphabet\matheur{bold}{U}{eur}{b}{n}

\numberwithin{equation}{section}

\newtheorem{theo}{Theorem}[section]
\newtheorem{prop}[theo]{Proposition}
\newtheorem{lemm}[theo]{Lemma}
\newtheorem{coro}[theo]{Corollary}

\theoremstyle{definition}

\newtheorem{defi}[theo]{Definition}
\newtheorem{exam}[theo]{Example}
\newtheorem{exas}[theo]{Examples}

\theoremstyle{remark}

\newtheorem{rema}[theo]{Remark}

\begin{document}\allowdisplaybreaks\frenchspacing 

\setlength{\baselineskip}{1.1\baselineskip}

\title{Abelian Sandpiles and Algebraic Models}

\author{Gabriel Strasser}

\address{Gabriel Strasser: Mathematics Institute, University of Vienna, Oskar-Morgenstern-Platz 1, A-1090 Vienna, Austria} \email{gabriel.strasser@univie.ac.at}
\begin{abstract} 
Motivated by the coincidence of topological entropies the connection between abelian sandpiles and harmonic models was established by K. Schmidt and E. Verbitskiy (2009). The dissipative sandpile models were shown to be symbolic representations of algebraic $\mathbb{Z}^d$-actions of the harmonic models. Both models are determined by so-called  simple sandpile polynomials. We extend this result to arbitrary sandpile polynomials. Moreover, we show that any sandpile model determined by a factor of a sandpile polynomial acts as an equal entropy cover of the corresponding algebraic model. For a special class of factors these covers are shown to be symbolic representations. 
\end{abstract}

\maketitle

\section{Introduction}\label{s:intro}
The abelian sandpile model (ASM) is a lattice model introduced by Dhar in \cite{4Dhar1}. It is inspired by an automaton model of self-organised criticality presented by Bak, Tang and Wiesenfeld (BTW) in \cite{BTW1,BTW2}. In both models grains of sand are dropped at sites in a finite set of the lattice thereby forming piles of sand. Eventually these piles get ``unstable''  and ``topple''. In the ASM the toppling conditions depend on the local heights of the piles rather than their gradients as in the BTW model: if a sandpile reaches or exceeds a critical height it gets unstable and topples according to a given rule. The toppling is formalised by applying so-called toppling operators. In contrast to the BTW model these operators commute in the ASM, motivating the name of the model. The abelian property led to extensive research on generalisations of Dhar's model, also named ASMs. For example, the model was studied on different lattices (\cite{Dhar3, Redig3}), the toppling rules were varied (\cite{Speer}), and extended to infinite volume (\cite{Jarai,4Redig2,Redig5}). Dhar also showed in \cite{Dhar2} that the topological entropy $\textup{h}(\sigma)$ of the shift $\sigma$ on the set $\mathcal{R}$ of infinite recurrent configurations of the ASM in $\mathbb{Z}^2$ (cf. \eqref{eq:infiniteba}) is given by
\[
\textup{h}(\sigma)=\int_0^1 \int_0^1 \log \left(4-2 \cos(2\pi x_1) -2 \cos(2 \pi x_2)\right) dx_1 dx_2.
\]
Motivated by the coincidence of topological entropies the connection between abelian sandpiles and algebraic models (cf. Section \ref{s:algebraicmodel}) was established in \cite{4Schmidt1}. Among others the question of uniqueness of the measure of maximal entropy is discussed. Although this uniquness conjecture remains unresolved the authors show in \cite{4Schmidt1} that the dissipative ASMs (cf. Definition \ref{d:dissipative}) in infinite volume, studied in \cite{4Redig2}, have a unique measure of maximal entropy. In fact they are symbolic representations of their corresponding algebraic models. 

In this paper we investigate the question whether there are other abelian sandpile models which act as symbolic representations of algebraic models. Taking into account the difficulties encountered in \cite{4Schmidt1} and described above we restrict ourselves to the case of dissipative ASMs. 
In Section \ref{s:ASM} we give a short description of a general dissipative ASM where the toppling condition is given as in Dhar's model and the toppling rule is determined by a wider class of Laurent polynomials, which we call sandpile polynomials (cf. Definition \ref{d:sandpilepol}). 
In the next section we introduce a class of algebraic models each of which is determined by an expansive Laurent polynomial (cf. Definition \ref{d:expansiveh}). In Section \ref{s:symbolicrepresentation} we take advantage of the fact that the methods used in \cite{4Redig2} and \cite{4Schmidt1} for dissipative ASMs are equally applicable to the more general versions. This means that the ASMs introduced in Section \ref{s:ASM} are symbolic representations of their algebraic counterparts defined in Section \ref{s:algebraicmodel}. In the last section we modify the toppling condition in the spirit of the original BTW model, where the toppling rule is defined similarly to Dhar's model but the toppling condition differs: A sandpile gets unstable and topples if the difference between heights of two neighbouring sites exceeds some critical threshold. In the modified version this linear dependence is described by an expansive Laurent polynomial $g$ and we want to find a suitable toppling rule, determined by an expansive Laurent polynomial $f$, such that the resulting model 
\begin{enumerate}
\item represents an abelian group (in finite volume) and
\item is a symbolic representation of the algebraic model associated to $f$ (in infinite volume),
\end{enumerate}
simulating the properties of Dhar's ASMs in finite and infinite volume. More precisely, we first show in Theorem \ref{t:equalentropy} that if $h=fg$ is a sandpile polynomial, where $f$ and $g$ determine the toppling rule and condition, respectively, then the modified BTW model is an equal entropy cover of the algebraic model associated to $f$. In the remainder of the last section we find a class of Laurent polynomials which satisfies the properties (1) and (2) from above, strengthening thereby the statement of Theorem \ref{t:equalentropy}. This result is summarised in Theorem \ref{t:symbolicrepresentationgeneral}. At the end we express our interest in the question whether there exist other classes of Laurent polynomials which fulfil these properties. 

In general, the concrete physical picture of sandpiles getting unstable and toppling might no longer be fitting in a BTW-like model. Although we show that these models can be realised, under specific conditions, as abelian subgroups of 'classical' ASMs described in Section \ref{s:ASM} they may also serve as toy models with implicit abelian structure for other physical phenomena. 

\section{The abelian sandpile model}\label{s:ASM}
Before describing the abelian sandpile model we need some preparation. 
\begin{defi}\label{d:dompol}
Let $d \ge 1$. We call a Laurent polynomial $h=\sum_{\mathbf{j} \in \mathbb{Z}^d} h_\mathbf{j} u^{\mathbf{j}} \in \mathbf{R}_d$ {\it lopsided} if there exists an $\mathbf{m} \in \mathbb{Z}^d$ such that 
\begin{equation*}
2h_\mathbf{m} > \sum_{\mathbf{j} \in \mathbb{Z}^d} | h_\mathbf{j}|>0. 
\end{equation*}
We denote the {\it dominant coefficient} $h_\mathbf{m}$ of $h$ by $\gamma_h$.
\end{defi}
\begin{defi}\label{d:sandpilepol}
 Let $d\ge1$. A lopsided Laurent polynomial $h \in \mathbf{R}_d$ with $\gamma_h=h_\mathbf{0}$ is called a {\it sandpile polynomial} if $h_\mathbf{j}\le 0$ for all $\mathbf{j} \in \mathbb{Z}^d \setminus \{\mathbf{0}\}$. We denote the collection of sandpile polynomials by $\mathbf{P}_d$. Moreover, $h \in \mathbf{P}_d$ is called {\it simple} if $h_{\mathbf{i}}=h_{-\mathbf{i}} \in \{-1,0\}$ for every $\mathbf{i}\in \mathbb{Z}^d\setminus \{\mathbf{0}\}$ and
\begin{equation*}\label{eq:hsimple}
h_{\mathbf{i}}\cdot h_{-\mathbf{i}}=
\begin{cases}
\gamma_h^2 & \textup{if} \enspace  \mathbf{i}=\mathbf{0}
\\
1& \textup{if} \enspace  \mathbf{i}\ \textup{or}\ \mathbf{-i}\ \textup{is a unit vector in}\ \mathbb{Z}^d
\\
0& \textup{otherwise.} 
\\
\end{cases}
\end{equation*} 
 \end{defi}
 \begin{rema}
 The position of the dominant coefficient $\gamma_h=h_{\mathbf{0}}$ in the definition of a sandpile polynomial serves only notational benefits. This condition will have no effect on the theory presented in this paper.
 \end{rema}
For $h \in \mathbf{P}_d$ and any finite subset $F \subseteq \mathbb{Z}^d$, in symbols $F \Subset \mathbb{Z}^d$, we define the so-called {\it toppling matrix} $\Delta=\Delta^{h,F}$ associated to $h$ by setting 
\begin{equation}\label{eq:topplingmatrix}
\Delta_{\mathbf{i}\mathbf{j}}={h}_{\mathbf{i}-\mathbf{j}}
\end{equation} 
for every $\mathbf{i},\mathbf{j} \in F$. Obviously, the matrix $\Delta$ has dimension $|F| \times |F|$ and it inherits for every $\mathbf{i},\mathbf{j} \in F$ the following properties from $h$:

\begin{equation*}
\begin{aligned} 
  \textup{(P1)}\ \ & \Delta_{\mathbf{i}\mathbf{i}}=\gamma_h > 0,\  \Delta_{\mathbf{i}\mathbf{j}} \le 0\ \textup{for}\ \mathbf{i}\neq \mathbf{j},
  \\
  \textup{(P2)}\ \ & \sum_{\mathbf{k} \in F} \Delta_{\mathbf{i}\mathbf{k}} >0.
  \end{aligned}
\end{equation*}

We give a short introduction to the ASM. The content of this subsection is taken mainly from \cite{Redig1}, \cite{4Redig2}, and \cite{Speer}, where the theory is treated in much greater detail. 

 Let $d \ge 1$, $h \in \mathbf{P}_d$, and denote by $\mathbb{Z}_+$ the set of nonnegative integers. Assume that $F \Subset \mathbb{Z}^d$ and write $\Delta=\Delta^{h,F}$ for the toppling matrix associated to $h$. A {\it configuration} is an element of $\mathbb{Z}_+^F$. We call a configuration $v \in \mathbb{Z}_+^F$ {\it stable} if $v_\mathbf{i} < \gamma_h$ for all $\mathbf{i} \in F$ and {\it unstable} otherwise. If $v$ is unstable at site $\mathbf{i}$, say, it {\it topples} at this site and results in another configuration $T_\mathbf{i}(v)\in \mathbb{Z}_+^F$ defined by 
\begin{equation}
\label{eq:topplingrule}
\smash[t] T_\mathbf{i}(v)_\mathbf{j}=(v-\Delta_{\cdot \mathbf{i}})_\mathbf{j}=
\begin{cases}
v_\mathbf{i}-\gamma_h & \textup{if} \enspace  \mathbf{j}=\mathbf{i}
\\
v_\mathbf{j} + |{h}_\mathbf{j-i}| & \textup{if} \enspace \mathbf{j}\neq \mathbf{i}\enspace \textup{and} \enspace \mathbf{j-i} \in \textup{supp}(h)
\\
v_\mathbf{j} & \textup{otherwise},
\end{cases}
\end{equation}
where $\textup{supp}(h)=\{\mathbf{j}\in \mathbb{Z}^d: h_\mathbf{j} \neq 0\}$ denotes the {\it support} of $h$.

The terminology is based on the physical picture of having grains of sand on a finite set $F$ of the lattice $\mathbb{Z}^d$. At each site of $F$ the grains can pile up to a given threshold $\gamma=\gamma_h$. The pile gets unstable when the threshold is reached or exceeded and topples according to the rule given by toppling matrix $\Delta$ associated to $h$. 

By the finiteness of $F$ and from the fact that at each toppling at least one grain is lost (cf. Definition \ref{d:sandpilepol} and \eqref{eq:topplingrule}) each unstable configuration $v$ has to stabilise after finitely many topplings. The stabilisation of $v$ in $F$ is denoted by $S(v)=S_F^{(h)}(v)$. Furthermore, the toppling operators $T_\mathbf{i}$ commute, i.e., $T_{\mathbf{i}}(T_\mathbf{j}(v))=T_{\mathbf{j}}(T_\mathbf{i}(v))$ for every configuration $v$ unstable at $\mathbf{i}, \mathbf{j} \in F$. Hence $S(v)$ is independent of the order of the toppling and therefore well-defined. We denote the space of all stable configurations by $\mathcal{S}_F=\mathcal{S}^{(h)}_F$ and note that $\mathcal{S}_F=\{0,\dots,\gamma_h-1\}^F$. 
 
The set $\mathcal{S}_{F}$ of stable configurations on $F$ forms a semigroup under the following addition. If $u,v \in \mathcal{S}_F$ then
 \begin{equation*}
 u \oplus v = S(u+v) \in \mathcal{S}_F,
 \end{equation*}
 where the configuration $w=u+v$ is defined by the usual coordinate-wise addition. The unique maximal subgroup $\mathcal{R}_F$ of $\mathcal{S}_F$ is called the set of {\it recurrent configurations} on $F$. The finite group $\mathcal{R}_F$ is abelian and therefore carries the unique addition-invariant Haar probability measure $\mu_F$. Set
\begin{equation*}
N_F(\mathbf{i})= -\sum_{\mathbf{j} \neq \mathbf{i}}\Delta^{h,F}_\mathbf{ji}
\end{equation*}
and observe that $N_F(\mathbf{i})< \Delta_{\mathbf{ii}}=\gamma_h$. For simple $h$ the number $N_F(\mathbf{i})$ describes the number of neighbours of the site $\mathbf{i}$ in $F$. As presented in the following the set $\mathcal{R}_F$ of recurrent configurations on $F$ are characterised by the so-called {\it burning algorithm}, which is due to Dhar and presented in \cite{Redig1} for simple $h$ and in \cite{Speer} for general $h\in \mathbf{P}_d$.

Put
\begin{equation*}
\mathcal{P}_F=\{v\in \mathcal{S}_F: v_\mathbf{i}\ge N_F(\mathbf{i})\ \textup{for at least one}\ \mathbf{i}\in F \},
\end{equation*}
then the set $\mathcal{R}_F$ of recurrent configurations is given by 
\begin{equation}\label{eq:finiteba}
\mathcal{R}_F= \bigcap_{\emptyset\neq E \subset F} \mathcal{P}_E. 
\end{equation}
Equation \eqref{eq:finiteba} extends immediately to infinite volume by defining

\begin{equation} \label{eq:infiniteba} 
\mathcal{R}=\mathcal{R}_{\mathbb{Z}^d}= \bigcap_{\emptyset\neq E \Subset \mathbb{Z}^d} \mathcal{P}_E.
\end{equation}
\begin{rema}
For a finite subset $E\subset F$ with not necessarily finite $F \subset \mathbb{Z}^d$ Equation \eqref{eq:finiteba} and \eqref{eq:infiniteba} implicitly uses the obvious embedding of $\mathcal{P}_E$ into $\mathcal{S}_F$ as collection of cylinder sets in $\mathcal{S}_F$. It follows that every $v \in \mathcal{R}$ can be approximated arbitrarily well by elements in $\mathcal{R}_F$ for suitable $F \subset \mathbb{Z}^d$ and it is easily seen that $\mathcal{R}$ is a perfect set. In particular, $\mathcal{R}$ is closed and therefore compact.
\end{rema} 
The compact set $\mathcal{R}\subset \{0,1,\dots, \gamma_h-1\}^{\mathbb{Z}^d}$ is called the {\it abelian sandpile model (ASM)}  in infinite volume. Obviously, the set $\mathcal{R}$ together with the $\mathbb{Z}^d$-shift action $\sigma:\mathcal{R}\to \mathcal{R}$, defined by
\begin{equation*}
\sigma^\mathbf{m}(v)_\mathbf{n}=v_{\mathbf{n+m}}
\end{equation*}
for every $\mathbf{m,n}\in \mathbb{Z}^d$, determine a $d$-dimensional shift space $(\mathcal{R},\sigma)$. We will not distinguish nominally between the ASM and its associated shift space.
\begin{rema}
Note that the abelian sandpile model $\mathcal{R}$ is determined by the sandpile polynomial $h \in \mathbf{P}_d$. To emphasise this dependence we denote the associated shift space by $(\mathcal{R}^{(h)},\sigma_h)$.
\end{rema}

\section{The algebraic model}\label{s:algebraicmodel}
We begin this section with some basic definitions. Let $d \ge 1$ and $X$ a compact metrisable space. A {\it continuous $\mathbb{Z}^d$-action} is a $\mathbb{Z}^d$-action $T: \mathbf{n}\to T^\mathbf{n}$ by homeomorphisms of $X$. 
If $S$ is a second $\mathbb{Z}^d$-action by homeomorphisms of a compact metrisable space $Y$, say, then $(Y,S)$, or simply $Y$, is a {\it factor} of $(X,T)$ if there exists a continuous surjective map $\phi:X \to Y$ such that 
\begin{equation}\label{eq:factor}
\phi \circ T^\mathbf{n} = S^\mathbf{n} \circ \phi
\end{equation}
for every $\mathbf{n} \in \mathbb{Z}^d$. If the {\it factor map} $\phi$ in \eqref{eq:factor} is a homeomorphism then it is called a {\it conjugacy} and $S$ and $T$ are {\it topologically conjugate}.

Let $\ell^1(\mathbb{Z}^d,\mathbb{Z})$ and $\ell^\infty(\mathbb{Z}^d,\mathbb{Z})$ be the subgroups of integer-valued functions of the Banach spaces $\ell^1(\mathbb{Z}^d,\mathbb{R})$ and $\ell^\infty(\mathbb{Z}^d,\mathbb{R})$ with norms $\|.\|_1$ and $\|.\|_\infty$, respectively. We write the elements $v \in \ell^\infty(\mathbb{Z}^d,\mathbb{R})$ as $v=(v_\mathbf{n})$ and define the shift-action $\bar{\sigma}$ of $\mathbb{Z}^d$ on $\ell^\infty(\mathbb{Z}^d,\mathbb{R})$ by
\begin{equation}\label{eq:4shift}
(\bar{\sigma}^\mathbf{m} v)_\mathbf{n}=v_\mathbf{m+n}
\end{equation}
for every $\mathbf{m,n}\in \mathbb{Z}^d$. The restriction of $\bar{\sigma}$ to a subshift $Y\subset \ell^\infty(\mathbb{Z}^d,\mathbb{Z})$ will be denoted by $\sigma_Y$ or simply $\sigma$ if the subshift is clear from the context. A subshift $Y$ is called {\it symbolic cover} of $(X,T)$ with {\it covering map} $\phi$, say, if $(X,T)$ is a factor of $(Y,\sigma)$. A symbolic cover $Y$ is an {\it equal entropy cover} if in addition the topological entropies $\textup{h}(\sigma)$ and $\textup{h}(T)$ coincide. For a symbolic cover $Y$ being a {\it symbolic representation} of $(X,T)$ we need to find a shift-invariant Borel set $Y'\subset Y$ with $\nu(Y')=1$ for every shift-invariant probability measure $\nu$ on $Y$ of maximal entropy such that the restriction of $\phi$ to $Y'$ is injective.

If $X$ is a compact abelian metrisable group with identity element $e_X$ and normalised Haar measure $\lambda_{X}$ then any $\mathbb{Z}^d$-action $\alpha:\mathbf{n} \to \alpha^\mathbf{n}$ by continuous automorphisms of $X$ is called an {\it algebraic $\mathbb{Z}^d$-action}. We say an algebraic $\mathbf{Z}^d$-action $\alpha$ on $X$ is {\it expansive} if there exists an open neighbourhood $\mathcal{O}$ of the identity $e_X$ of $X$ with $\bigcap_{\mathbf{n}\in \mathbb{Z}^d} \alpha^{\mathbf{n}}(\mathcal{O})=\{e_X\}$. A point $x \in X$ is called {\it homoclinic} under $\alpha$ if $\lim_{|\mathbf{n}|\to \infty} \alpha^{\mathbf{n}} x=0$.

Let $\mathbf{R}_d=\mathbb{Z}[u_1^{\pm 1},\dots,u_d^{\pm 1}]$ be the ring of Laurent polynomials with integer coefficients in $d$ variables. Put $u^\mathbf{m}=u_1^{m_1}\cdots u_d^{m_d}$ for every $\mathbf{m}=(m_1,\dots,m_d) \in \mathbb{Z}^d$ and write each $h \in \mathbf{R}_d$ as
\begin{equation}\label{eq:h} 
h = \sum_{n\in \mathbb{Z}^d} h_\mathbf{n} u^{\mathbf{n}}, 
\end{equation}
where $h_{\mathbf{n}} \in \mathbb{Z}$ and $h_\mathbf{n}= 0$ for all but finitely many $\mathbf{n} \in \mathbb{Z}^d$. Obviously we can identify each $h \in \mathbf{R}_d$ of the form \eqref{eq:h} with $(h_\mathbf{n})\in \ell^1(\mathbb{Z}^d,\mathbb{Z})$, thereby obtaining an isomorphism $\mathbf{R}_d \cong \ell^1(\mathbb{Z}^d,\mathbb{Z})$.  

If $\alpha$ denotes the shift-action of $\mathbb{Z}^d$ on $\mathbb{T}^{\mathbb{Z}^d}$, defined similarly to \eqref{eq:4shift}, then every nonzero $h=\sum_{\mathbf{m}\in \mathbb{Z}^d} h_\mathbf{n} u^\mathbf{n} \in \mathbf{R}_d$ defines a continuous surjective group homomorphism
\begin{equation*}
h(\alpha)=\sum_{\mathbf{n}\in \mathbb{Z}^d} h_\mathbf{n}\alpha^\mathbf{n}: \mathbb{T}^{\mathbb{Z}^d} \to\mathbb{T}^{\mathbb{Z}^d}.
\end{equation*}
Consider the corresponding compact abelian group
\begin{equation}\label{eq:Xh}
X_h=\Big\{ (x_\mathbf{n}) \in \mathbb{T}^{\mathbb{Z}^d}: \sum_{\mathbf{n}\in \mathbb{Z}^d} h_\mathbf{n} x_\mathbf{n+m}=0\ \textup{for all}\ \mathbf{m}\in \mathbb{Z}^d \Big\}=\ker h(\alpha)
\end{equation}
and denote by $\alpha_h$ the restriction of $\alpha$ to $X_h$. By the term {\it algebraic model} determined by $h$ we refer to the dynamical system $(X_h,\alpha_h)$.
\begin{defi}\label{d:expansiveh}
We say a Laurent polynomial $h \in\mathbf{R}_d$ is {\it expansive} if it has no zeros in $\mathbb{S}^d$, where $\mathbb{S}=\{ s \in \mathbb{C}: |s|=1 \}$. 
\end{defi}
By \cite[Thm. 6.5]{4Schmidt2} the shift-action $\alpha_h$ is expansive if and only if $h$ is expansive.
\begin{rema}\label{r:dompolexpansive}
Note that any lopsided polynomial is expansive. In particular every sandpile polynomial is expansive (cf. Definition \ref{d:dompol} and \ref{d:sandpilepol}). 
\end{rema} 
The entropy $\textup{h}(\alpha_h)$ of $\alpha_h$ is given by the logarithmic Mahler measure (cf. \cite{4LindSchmidtWard})
\[
\textup{h}(\alpha_h)= \int_0^1 \dots \int_0^1 \log|h(e^{2\pi it_1},\dots,e^{2\pi it_d})|\ dt_1\dots dt_d
\]
and coincides with the metric entropy $\textup{h}_{\lambda_{X_h}}(\alpha_h)$ of $\alpha_h$ with respect to the normalised Haar measure $\lambda_{X_h}$ on $X_h$. The measure $\lambda_{X_h}$ is the unique measure of maximal entropy if $h$ is expansive. 

Define the surjective map $\rho: \ell^\infty(\mathbb{Z}^d,\mathbb{R})\to \mathbb{T}^{\mathbb{Z}^d}$ by
\begin{equation}\label{eq:rho}
\rho(v)_\mathbf{n}=v_\mathbf{n} \pmod 1
\end{equation}
for every $v=(v_\mathbf{n}) \in \ell^\infty(\mathbb{Z}^d,\mathbb{R})$ and $\mathbf{n} \in \mathbb{Z}^d$.
If $h=\sum_{\mathbf{m}\in \mathbb{Z}^d} h_\mathbf{n} u^\mathbf{n} \in \mathbf{R}_d$ we put
\begin{equation*}
\tilde{h}=\sum_{\mathbf{m}\in \mathbb{Z}^d} h_\mathbf{n} u^{-\mathbf{n}} 
\end{equation*}
and note that 
\begin{equation}\label{eq:hproperties}
\rho \circ h(\bar{\sigma})= h(\alpha) \circ \rho \quad \textup{and}\quad h(\bar{\sigma})(v) =\tilde{h} \cdot v,
\end{equation}
for every $v \in \ell^1(\mathbb{Z}^d,\mathbb{Z})\cong\mathbf{R}_d$, where  $(\tilde{h}\cdot v)_{\mathbf{n}}=\sum_{\mathbf{j}}\tilde{h}_jv_{\mathbf{n-j}}$ for every $\mathbf{n}\in \mathbb{Z}^d$. Finally we remark that the product $u\cdot w$ and the corresponding identity on the right-hand side of \eqref{eq:hproperties} are well-defined whenever $u \in \ell^1(\mathbb{Z}^d,\mathbb{R})$ and $w \in \ell^\infty(\mathbb{Z}^d,\mathbb{R})$. 

\section{The ASM $\mathcal{R}^{(h)}$ as symbolic representation of $X_{\tilde{h}}$}\label{s:symbolicrepresentation}
Fix $d\ge 1$ and a sandpile polynomial $h \in \mathbf{P}_d$. Recall that $h$ is expansive (cf. Remark \ref{r:dompolexpansive}). According to the proof of Lemma 4.5 in \cite{4LindSchmidt} the equation 
\begin{equation}\label{eq:homoclinicequation2} 
h\cdot w=\tilde{h}(\bar{\sigma})(w) =\delta_\mathbf{0},
\end{equation}
where $\delta_{\mathbf{0}}$ is the indicator function of $\{\mathbf{0}\}$, has a unique solution $w=w^h \in \ell^1(\mathbb{Z}^d,\mathbb{R})$ if $h$ is expansive. The $\ell^1$-summability of the homoclinic point $w^h$ allows us to define group homomorphisms $\bar{\xi}_h: \ell^{\infty}(\mathbb{Z}^d,\mathbb{Z})\to \ell^{\infty}(\mathbb{Z}^d,\mathbb{R})$ and $\xi_h: \ell^{\infty}(\mathbb{Z}^d,\mathbb{Z})\to \mathbb{T}^{\mathbb{Z}^d}$, given by
\begin{equation}\label{eq:xih}
 \bar{\xi}_h(v)= w^h \cdot v  \qquad \textup{and}\qquad \xi_h(v)= (\rho \circ \bar{\xi})(v).
\end{equation}
The maps $\bar{\xi}$ and $\xi$ satisfy the following properties (cf. \cite[Prop. 2.3]{Einsiedler}).
\begin{prop} Let $h\in \mathbf{R}_d$ be expansive. Then we have for any $v \in \ell^\infty(\mathbb{Z}^d,\mathbb{Z})$ that
\begin{equation}\label{eq:hinverse}
\tilde{h}(\bar{\sigma})(\bar{\xi}_h(v))= \bar{\xi}_h(\tilde{h}({\sigma})(v))=v.
\end{equation}
Furthermore, $\xi_h$ satisfies 
\begin{equation}\label{eq:kernel}
\begin{split}
\xi_h \circ \sigma^\mathbf{n} &= \alpha^\mathbf{n} \circ \xi_h\quad \textup{for every}\ \mathbf{n}\in \mathbb{Z}^d, \\
\ker \xi_h &= \tilde{h}(\sigma)(\ell^\infty(\mathbb{Z}^d,\mathbb{Z})), \\
\ker \xi_h \cap \ell^1(\mathbb{Z}^d,\mathbb{Z})&=\tilde{h}(\sigma)(\ell^1(\mathbb{Z}^d,\mathbb{Z}))=h\cdot \mathbf{R}_d.
\end{split}
\end{equation}
\end{prop}
Following the penultimate section in \cite{4Schmidt1} one can apply the same methods used therein to get analogous results for general $h\in \mathbf{P}_d$, as stated below. 

\begin{lemm}\label{l:propxi}
 Let $h \in \mathbf{P}_d$. Then 
\begin{enumerate}
  \item $\xi_h(\mathcal{R}^{(h)})=X_{\tilde{h}}$.
  \item $\xi_h(v) \neq \xi_h(w)$ for all $v,w \in \mathcal{R}^{(h)}$ with $v-w \in \mathbf{R}_d$.
  \item If $v \in \ell^\infty(\mathbb{Z}^d,\mathbb{Z})$ then there exists an $m \in \ell^\infty(\mathbb{Z}^d,\mathbb{Z})$ such that $w=v+h\cdot m \in \mathcal{R}^{(h)}$.
 \end{enumerate}
 Furthermore, the topological entropies of $\alpha_{\tilde{h}}$ on $X_{\tilde{h}}$ and $\sigma_{\mathcal{R}^{(h)}}$ on $\mathcal{R}^{(h)}$ coincide.
\end{lemm}  
In particular the above Lemma shows that $\mathcal{R}^{(h)}$ is an equal entropy cover of $X_{\tilde{h}}$. This statement is further improved in \cite[Thm. 6.6]{4Schmidt1},
where it is shown that $\mathcal{R}^{(h)}$ is a symbolic representation of $X_{\tilde{h}}$:
\begin{theo} \label{t:symbolicrepresentation}
For $h \in \mathbf{P}_d$
the subshift $\mathcal{R}^{(h)}$ admits a unique measure $\mu$ of maximal entropy for which the covering map $\xi_h$ restricted to $\mathcal{R}^{(h)}$ is almost one-to-one.
\end{theo}
\begin{rema}\label{r:analogy}
The second statement of the above Theorem is based on \cite[Prop. 3.2 and Thm. 3.1]{4Redig2}. In this paper the authors construct the measure $\mu$ on $\mathcal{R}^{(h)}$ as unique accumulation point of any sequence $\mu_{j}$ of Haar measures on $\mathcal{R}^{(h)}_{F_j}$ with $F_j \Subset \mathbb{Z}^d$ and $F_j \uparrow \mathbb{Z}^d$. The convergence is achieved by two properties. First, the existence of the unique Haar measures $\mu_{j}$ on $\mathcal{R}^{(h)}_{F_j}$, and second, from the so-called {\it dissipativity} of the ASM $\mathcal{R}^{(h)}$.
\end{rema} 
\begin{defi}\label{d:dissipative}
An abelian sandpile model $\mathcal{R}^{(h)}$ determined by the sandpile polynomial $h\in \mathbf{P}_d$ is called {\it dissipative} if there exists a (unique) $w^h \in \ell^1(\mathbb{Z}^d,\mathbb{Z})$ such that $h\cdot w^h=\delta_\mathbf{0}$.
\end{defi}
The first property follows from the fact that $\mathcal{R}^{(h)}_{F_j}$ is a compact abelian group for any $F_j\Subset \mathbb{Z}^d$ and the second from $h$ being expansive (cf. Remark \ref{r:dompolexpansive} and \eqref{eq:homoclinicequation2}). These are the crucial properties which allow to prove the extension of the statements of Lemma \ref{l:propxi} and Theorem \ref{t:symbolicrepresentation} from simple $h$ (for which they were originally proved) to general $h\in \mathbf{P}_d$ in a completely analogous way. 

\section{Generalisations}\label{s:generalisations}
So far we have seen that for any sandpile polynomial $h \in \mathbf{P}_d, d \ge 1$ and $F \Subset \mathbb{Z}^d$
\begin{enumerate}
\item the recurrent configurations $\mathcal{R}_F^{(h)}$ form an abelian group and
\item $\mathcal{R}^{(h)}$ is a symbolic representation of $X_{\tilde{h}}$.
\end{enumerate}
In this section we want to find other classes of polynomials which fulfil the above properties. Before proving our results we consider two simple examples in order to demonstrate the idea of our approach. For property (1) we choose the one-dimensional BTW model, mentioned in Section \ref{s:intro}, in finite volume $F$ as example and show how to realise it as abelian subgroup of an ASM $\mathcal{R}_F$. Similar to the transition of Dhar's model from finite to infinite volume, presented in \cite{4Redig2}, we use a sequence of measures on the abelian groups in finite volume $F$ to find the unique measure of maximal entropy in infinite volume, by taking the limit along $F \uparrow \mathbb{Z}^d$. Nevertheless, we first consider both problems separately. To avoid difficulties of non-dissipativity in infinite volume (cf. Section \ref{s:intro}) we need to vary the example from above when investigating property (2). 

\subsection{Motivation}
\begin{exam} \label{ex:BTWex}
For the description of the one-dimensional BTW model we fix $N \ge 1$, consider the finite set $F=\{1,\dots,N\}\subseteq \mathbb{Z}$, the Laurent polynomials $f=1-u,g= -u^{-1}+1 \in \mathbf{R}_1$ and the corresponding ($|F|\times|F|$)-matrices $\Delta^f,\Delta^g$ defined as in \eqref{eq:topplingmatrix}. Finally we set $\Delta'=\Delta^g  \Delta^f$ and note that for $ i,j \in F$ 
\[ 
\Delta'_{ij}=
\begin{cases}
2 & \textup{if} \enspace i=j \neq N \\
1 & \textup{if} \enspace i=j = N \\
-1 & \textup{if} \enspace |i-j|=1 \\
 0 & \textup{otherwise}.
\end{cases} 
\]
The BTW model is also a sandpile model but the toppling condition differ from those in Dhar's model. Sandgrains are dropped onto sites in $F$ thereby creating piles. A configuration $v \in \mathbb{Z}_+^F$, where $v_n\ge 0$ represents the height of the sandpile at site $n \in F$, is said to be {\it stable} if for every $n\in F$
\begin{equation*}
(\Delta^g  v)_n=\left.
\begin{cases} 
v_n-v_{n+1} & \textup{if}\enspace 1 \le n < N 
\\ v_n & \textup{if}\enspace n=N 
\end{cases} \right\}
< \Delta'_{nn},
\end{equation*} i.e., if the height of a sandpile differs from the height of its right neighbour by less than 2 and the sandpile at the right boundary has height zero. Obviously, $v$ gets {\it unstable} at site $k$, say, if 
\begin{equation}\label{eq:BTWunstable}
(\Delta^g  v)_k \ge \Delta'_{kk}.
\end{equation}
The unstable site $k$ then topples and one grain tumbles to the lower level to the right. At the boundary site $k=N$ the grain leaves the system. Formally, the configuration $T_k(v)$ after the toppling is given by 
\begin{equation}\label{eq:BTWtoppling}
T_k(v)_j=(v-\Delta^f_{\cdot k})_j.
\end{equation}
The toppling process will terminate as grains can leave the system. One of our interests in such a system lies in its so-called {\it recurrence class}. This means we want to understand how the system evolves when we keep adding sandgrains to it. The solution to this question is already mentioned in \cite{BTW1}: the system tends to a unique state, namely to the unique configuration $w$ with the property that $(\Delta^g  w)_n=\Delta'_{nn}-1, n \in F$, or equivalently, $w_n=N-n$ for every $n\in F$.
\end{exam}

We now highlight the somewhat hidden structure in Example \ref{ex:BTWex}. First note that the matrix $\Delta'$ defines a more general toppling matrix than we have introduced at the beginning of Section \ref{s:ASM}. More precisely, the condition (P2) is not satisfied, since $\sum_{j\in F}\Delta'_{ij}=0$ for $2\leq i \le N$. As shown in \cite{Speer} the matrix $\Delta'$ still determines an ASM $\mathcal{R}_F$ on $F$ with $\mathcal{R}_F\subset \mathcal{S}_F=\{ v \in \mathbb{Z}_+^F: 0 \le v_i < \Delta'_{ii}\ \textup{for all}\ i \in F \}$, and the set $\mathcal{R}_F$ is a finite abelian group characterised by the burning algorithm presented in \eqref{eq:finiteba}. Applying it to this special example allows us to identify the group $\mathcal{R}_F$ as singleton, whose only element $\Delta^g w$, given above, happens to be a multiple of $\Delta^g$. Another way to look at the last observation would be that the multiples of $\Delta^g$ form an abelian subgroup 
\begin{equation}\label{eq:WF}
\mathcal{W}_F=\{ \Delta^g  v \in \mathcal{R}_F:v \in \mathbb{Z}^F\} 
\end{equation} 
of $\mathcal{R}_F$. Neglecting the triviality of this statement, for the moment, we derive from this that the set
\begin{equation}\label{eq:VF}
\mathcal{V}_F=\{ v \in \mathbb{Z}^F: \Delta^g v \in \mathcal{R}_F \} \end{equation}
also defines an abelian group isomorphic to $\mathcal{W}_F$, since $\det(\Delta^g)\neq 0$. As we have seen, the group $\mathcal{V}_F$ arises as recurrence class of the dynamics driven by the laws of the toppling condition \eqref{eq:BTWunstable} and rule \eqref{eq:BTWtoppling}. A nontrivial example of the sets $\mathcal{W}_F$ and $\mathcal{V}_F$ is given in Remark \ref{r:finitecase}.

One goal of this section is to find classes of Laurent polynomials $f$ and $g$ which allow the realisation of the heuristic idea described above. Depending on the Laurent polynomial $f$ and $g$ these BTW-like models potentially provide a variety of physical toy models with implicit abelian structure. 
 
With regards to the above discussion in finite volume the analogous idea in infinite volume would be to consider a Laurent polynomial $h \in \mathbf{P}_d, d \ge 1$ of the form $h=fg$ with $f,g \in \mathbf{R}_d$ and to determine if the multiples of $g$ in $\mathcal{R}^{(h)}$ serve as a symbolic representation for $X_{\tilde{f}}$. But before we get into action we familiarise ourselves with that approach by first preparing the proper setting and then investigating an easy example. 

Fix $d \ge 1$ and let $h=fg\in \mathbf{P}_d$ with $f,g \in \mathbf{R}_d$. Note that since $h$ is expansive, $f$ and $g$ are as well. Moreover, the homoclinic point $w^{h}$ of $h$ satisfies $w^{h}=w^{f}\cdot w^{g}$, where $w^{f}$ and $w^{g}$ are the homoclinic points of $f$ and $g$, respectively (cf. \eqref{eq:homoclinicequation2}). Therefore and by \eqref{eq:xih} we obtain $\bar{\xi}_h=\bar{\xi}_f\circ \bar{\xi}_g$ and $\xi_h=\rho\circ (\bar{\xi}_f\circ \bar{\xi}_g)$. Recall from \eqref{eq:kernel} that $\ker\xi_g=g\cdot \ell^{\infty}(\mathbb{Z}^d,\mathbb{Z})$ for any expansive $g \in \mathbf{R}_d$. From the previous section we know that $\mathcal{R}^{(h)}$ is a symbolic representation of $X_{\tilde{h}}$ via the covering map $\xi_h$. Since $X_{\tilde{f}}  \subset X_{\tilde{h}}$ it is tempting to argue that there is also a symbolic representation of $X_{\tilde{f}}$ within $\mathcal{R}^{(h)}$. We take a look at 
\begin{equation}\label{eq:Wgh}
\mathcal{W}^{(h)}_g=\mathcal{R}^{(h)} \cap \ker \xi_g=\{ g\cdot u \in \mathcal{R}^{(h)}: u \in \ell^\infty(\mathbb{Z}^d,\mathbb{Z})\}
\end{equation} 
and observe that $\xi_h(\mathcal{R}^{(h)} \cap \ker \xi_g)=\xi_f(\mathcal{V}^{(h)}_g)\subseteq X_{\tilde{f}}$, where
\begin{equation}\label{eq:Vgh}
\mathcal{V}^{(h)}_g=\bar{\xi_g}(\mathcal{W}^{(h)}_g)=\{u \in \ell^\infty(\mathbb{Z}^d,\mathbb{Z}): g\cdot u \in \mathcal{R}^{(h)}\}.
\end{equation}

We will show that the set $\mathcal{W}^{(h)}_g$ of multiples of $g$ in $\mathcal{R}^{(h)}$ is indeed an equal entropy cover of $X_{\tilde{f}}$. In order to find the unique measure of maximal entropy on $\mathcal{W}^{(h)}_g$ we depend on the existence of the Haar measures on the finite abelian subgroups $\mathcal{W}_F$ (cf. \eqref{eq:WF}) for $F\Subset \mathbb{Z}^d$. Under specific conditions we can prove that the BTW-like model $\mathcal{W}_F$ in finite volume extends to the symbolic representation $\mathcal{W}^{(h)}_g$ of $X_{\tilde{f}}$ when $F \uparrow \mathbb{Z}^d$.   

In order to get a better understanding of $\mathcal{W}^{(h)}_g$ we first study $\mathcal{V}^{(h)}_g$. To this end we start with a simple example in dimension $d=1$. 
\begin{exam}\label{e:BTWlikeex}
Let $f=-u^{-1}+2,\ g=2-u \in \mathbf{R}_1$ and $h=fg = -2u^{-1}+5 -2u \in \mathbf{P}_1$. We want to determine $\mathcal{V}^{(h)}_g=\bar{\xi}_g(\mathcal{R}^{(h)}_\infty \cap \ker \xi_g)\subseteq \ell^\infty(\mathbb{Z},\mathbb{Z})$. Since $d=1$ and $h$ is simple we can easily deduce from the burning algorithm \eqref{eq:infiniteba} in infinite volume that each recurrent configuration in $\mathcal{R}^{(h)}$ can be described as a sequence of labels of a bi-infinite path in the following directed graph $G^{(h)}$:
\begin{center}
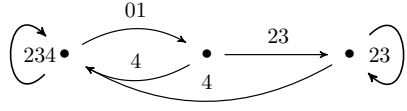

\scalebox{.75}{
	\begin{tikzpicture}[->,>=stealth',shorten >=2pt,auto,node distance=25mm, inner sep=2mm, semithick]
\node[] (1) {$\bullet$} {};
\node[] (2) [right of =1] {$\bullet$} {};
\node[] (3) [right of =2] {$\bullet$} {};
\Loop[dist=1.2cm,dir=WE,label=$234$,labelstyle=right](1) 
\Loop[dist=1.2cm,dir=EA,label=$2 3$,labelstyle=left](3)
\path
(1) edge [bend left] node[above] {$01$} (2)
(2) edge [bend left] node[above] {$4$} (1)
(2)	edge node[above] {$2 3$} (3)
(3) edge [bend left] node[above] {$4$} (1);
	\end{tikzpicture}}
	\captionof{figure}[\ The graph representing the configurations in $\mathcal{R}^{(h)}$.]{The graph representing the configurations in $\mathcal{R}^{(h)}$.}
\vspace{-3mm}
	\end{center}
\begin{rema}\label{r:forbidden}
We immediately see that the word $00$ does not occur in $G^{(h)}$ and hence is forbidden in $\mathcal{R}^{(h)}$. A more general observation is that whenever the symbol 0 or 1 shows up at site $n$, say, then we must see the symbol $4$ at some site $k > n$ before we could see 0 or 1 again. Obviously, all other finite words in the full 5-shift $\{0,\dots,4\}^\mathbb{Z}$ which respect this law are allowed.
\end{rema}
Clearly, $v \in \mathcal{V}^{(h)}_g$ if and only if $g \cdot v \in \mathcal{R}^{(h)}$. Note that 
\begin{equation}\label{eq:E0}
0\le (g \cdot v)_n=2 v_n - v_{n-1} \le 4
\end{equation}
for every $n\in \mathbb{Z}$. 
For $v \in \ell^\infty(\mathbb{Z},\mathbb{Z})$ we set 
\begin{equation*}
\| v\|^+_\infty= \sup\{ |v_n|: v_n >0 \}\qquad \textup{and}\qquad \|v\|^-_\infty= \sup\{ |v_n|: v_n <0 \}.
\end{equation*}
We first claim that if $u \in \mathcal{V}^{(h)}_g$ then $1 \le u_n \le 4$ for every $n \in \mathbb{Z}$. Observe that if $u_n=\|u\|^+_\infty$ for some $n \in \mathbb{Z}$ then $4 \ge(g\cdot u)_n = 2 \|u\|_\infty^+ -u_{n-1} \ge \|u\|_\infty^+$. Similarly, if $u_n = -\|u\|^-_\infty$ then $0 \ge -(g\cdot u)_n = 2 \|u\|_\infty^- + u_{n-1} \ge \|u\|_\infty^- \ge 0$. This proves that $0\le u_n \le 4$ for all $n \in \mathbb{Z}$ if $u \in \mathcal{V}^{(h)}_g$. Suppose $u \in \mathcal{V}^{(h)}_g$ such that there exists an $n \in \mathbb{Z}^d$ with $u_n=0$. Then, by \eqref{eq:E0}, $u_{n-1}=0$. Repeating this argument at site $n-1$ we see that $(g \cdot u)_n=(g\cdot u)_{n-1}=0$. In other words $g\cdot u$ contains the forbidden string 00 and hence does not lie in $\mathcal{R}^{(h)}$, which contradicts the fact that $u \in \mathcal{V}^{(h)}_g$. Our next claim is that each element in $\mathcal{V}^{(h)}_g$ is described precisely by a sequence of labels of a bi-infinite path in the following directed graph $G^{(h)}_g$:

\begin{center}
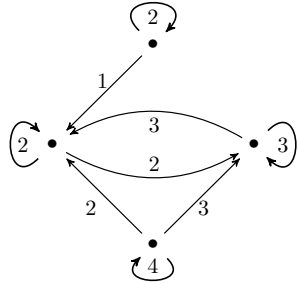

\scalebox{.75}{
\begin{tikzpicture}[->,>=stealth',shorten >=2pt,auto,node distance=25mm, inner sep=1mm, semithick, main node/.style={circle,draw}] 
\node[] (A) {$\bullet$}; \node (B) [below right of=A] {$\bullet$}; \node[circle,] (C) [below left of=A] {$\bullet$}; \node[] (D) [below left of=B] {$\bullet$};
\Loop[dist=1cm,dir=NO,label=$2$,labelstyle=below](A) 
\Loop[dist=0.99cm,dir=WE,label=$2$,labelstyle=right](C) 
\Loop[dist=0.99cm,dir=EA,label=$3$,labelstyle=left](B)
\Loop[dist=0.8cm,dir=SO,label=$4$,labelstyle=above](D); 
\path 
(A) edge node[above] {1} (C) 
(B) edge [bend right] node {3} (C)
(C) edge [bend right] node[above] {2} (B) 
(D) edge node {2} (C) 
(D) edge node[below] {3} (B);
	\end{tikzpicture}}
	\captionof{figure}[\ The graph $G^{(h)}_g$ representing the configurations in $\mathcal{V}^{(h)}_g$.]{The graph $G^{(h)}_g$ representing the configurations in $\mathcal{V}^{(h)}_g$.}
\end{center}
To see this we look at all possible values for $u \in \mathcal{V}^{(h)}_g$ at some site $n \in \mathbb{Z}$ and their impact on the values of the neigbouring sites according to \eqref{eq:E0}. If $u \in \mathcal{V}^{(h)}_g$ then the following conditions are satisfied:
\begin{eqnarray}\label{eq:E1}
u_n=1 &\Rightarrow & u_{n-1} \in \{1,2\}  \wedge u_{n+1} \in \{1,2\}\wedge (g\cdot u)_n \in \{0,1\}, \\ \label{eq:E2}
u_n=2 &\Rightarrow & u_{n-1} \in \{1,2,3,4\}  \wedge u_{n+1} \in \{1,2,3\}\wedge (g\cdot u)_n \in \{0,1,2,3\}, \\ \label{eq:E3}
u_n=3 &\Rightarrow & u_{n-1} \in \{2,3,4\}  \wedge u_{n+1} \in \{2,3\} \wedge (g\cdot u)_n \in \{2,3,4\}\\ \label{eq:E4}
u_n=4 &\Rightarrow & u_{n-1} =4  \wedge u_{n+1} \in \{2,3,4\}\wedge (g\cdot u)_n =4.
\end{eqnarray} 
These conditions are necessary to guarantee that $0 \le (g\cdot u)_n \le 4$. Taking further into account that $g\cdot u$ must lie in $\mathcal{R}^{(h)}$ we see that if $u_n=1$ then $u_k=2$ for all $k<n$. Indeed, suppose that $u_n=1$ and $l=\min\{j <n:u_j \ne 2\}$ exists. If $u_l =1$ then $(g\cdot u)_n\in \{0,1\}$, $(g\cdot u)_j\ne 4$ for all $l<j<n$ and $(g\cdot u)_{l} \in \{0,1\}$, by (\ref{eq:E1}) and (\ref{eq:E2}). If $u_l >2$ (hence $l\neq n-1$ and $u_{l+1}=2$ by (\ref{eq:E3}) and (\ref{eq:E4})) then $(g\cdot u)_n=0$, $(g\cdot u)_j\ne 4$ for all $l+1<j<n$ and $(g\cdot u)_{l+1} \in \{0,1\}$. In either case $g\cdot u$ contains a forbidden word in $\mathcal{R}^{(h)}$ as described in Remark \ref{r:forbidden}, which is impossible. Hence we can replace (\ref{eq:E1}) by
\begin{equation}\label{eq:E5}
u_n=1 \Rightarrow u_k=2\ \textup{for all}\ k<n \wedge u_{n+1}=2 \wedge (g\cdot u)_n \in \{0,1\}
\end{equation}
Examining carefully (\ref{eq:E2})--(\ref{eq:E5}) we recognise that the graph $G^{(h)}_g$ represents all configurations satisfying these conditions. Conversely, by a little exercise one is easily convinced that for every configuration $u$ obtained from the graph $G^{(h)}_g$ the configuration $g\cdot u$ does not contain any forbidden subconfigurations (cf. Remark \ref{r:forbidden}), so that $g\cdot u \in \mathcal{R}^{(h)}$, or equivalently, $u \in \mathcal{V}^{(h)}_g$. It follows that each $u\in \mathcal{V}^{(h)}_g$ is given by a sequence of labels of a bi-infinite path in the graph $G^{(h)}_g$, and vice versa. 

Since $\bar{\xi}_g$ restricted to $\mathcal{W}_g^{(h)}$ is a homeomorphism, with $g$ as inverse, and $\bar{\xi}_g(\mathcal{W}_g^{(h)})=\mathcal{V}^{(h)}_g$ there is no reason to favour $\mathcal{W}^{(h)}_g=g\cdot\mathcal{V}^{(h)}_g=\mathcal{R}^{(h)} \cap \ker \xi_g$ over $\mathcal{V}^{(h)}_g$. For the sake of completeness we give a description of configurations in $\mathcal{W}^{(h)}_g$. As one can easily deduce from $G^{(h)}_g$ the graph $G'^{(h)}_g$ representing the configurations in $\mathcal{W}^{(h)}_g$ is given by 
\begin{center}
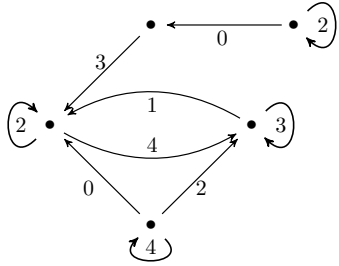

\scalebox{.75}{
\begin{tikzpicture}[->,>=stealth',shorten >=2pt,auto,node distance=25mm, inner sep=1mm, semithick, main node/.style={circle,draw}] 
\node[] (A) {$\bullet$}; \node (B) [below right of=A] {$\bullet$}; \node[circle,] (C) [below left of=A] {$\bullet$}; \node[] (D) [below left of=B] {$\bullet$}; \node[] (E) [right of=A] {$\bullet$};
\Loop[dist=1cm,dir=EA,label=$2$,labelstyle=left](E) 
\Loop[dist=0.99cm,dir=WE,label=$2$,labelstyle=right](C) 
\Loop[dist=0.99cm,dir=EA,label=$3$,labelstyle=left](B)
\Loop[dist=0.8cm,dir=SO,label=$4$,labelstyle=above](D); 
\path 
(A) edge node[above] {3} (C) 
(B) edge [bend right] node {1} (C)
(C) edge [bend right] node[above] {4} (B) 
(E) edge [] node {0} (A) 
(D) edge node {0} (C) 
(D) edge node[below] {2} (B);
	\end{tikzpicture}}
	\captionof{figure}[\ The graph $G'^{(h)}_g$ representing the configurations in $\mathcal{W}^{(h)}_g$.]{The graph $G'^{(h)}_g$ representing the configurations in $\mathcal{W}^{(h)}_g$.}
\end{center}
\end{exam}

\begin{rema}\label{r:finitecase}
Later on we want to construct, for particular Laurent polynomials $f,g$ and $h=fg$, a unique measure of maximal entropy on $\mathcal{V}^{(h)}_g$ (resp. $\mathcal{W}^{(h)}_g$) as limit of measures on $\mathcal{V}_F$ (resp. $\mathcal{W}_F$), for $F \Subset \mathbb{Z}^d$, where $\mathcal{V}_F$ (resp. $\mathcal{W}_F$) are defined in \eqref{eq:VF} (resp. \eqref{eq:WF}). In spite of this it might be useful to have at least once a concrete description of the finite abelian groups $\mathcal{V}_F$ and $\mathcal{W}_F$. 

For $f$ and $g$ as in Example \ref{e:BTWlikeex} and $F=\{1,\dots,N\}$ the configurations of the sets $\mathcal{V}_F$ and $\mathcal{W}_F$ can be represented by paths of length $N$ starting at $\star$ in the following graphs in Figure \ref{fig:VFandWF}. These sets are found analogously to Example \ref{ex:BTWex} and by applying similar considerations used to construct the graphs $G^{(h)}_g$ and $G'^{(h)}_g$.
\medskip 

\begin{center}
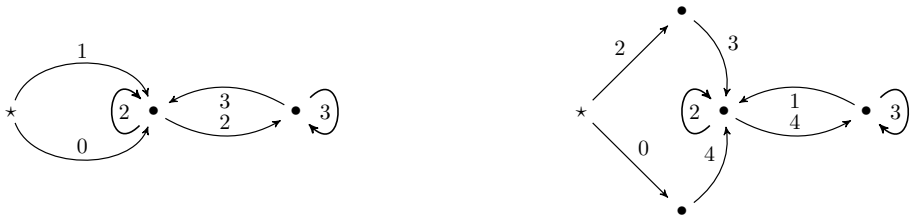

\scalebox{.75}{
\begin{tikzpicture}[->,>=stealth',shorten >=2pt,auto,node distance=25mm, inner sep=1mm, semithick, main node/.style={circle,draw}] 
\node[] (A) {$\star$};
\node (C) [right of=A] {$\bullet$}; 
\node (B) [right of=C] {$\bullet$};
\node (Z) [right of=B] {};  
\node (E) [right of=Z] {$\star$};
\node (F) [right of=E] {$\bullet$};
\node (G) [right of=F] {$\bullet$};
\node (H) [below right of=E] {$\bullet$}; 
\node(I) [above right of=E] {$\bullet$};   
\Loop[dist=0.99cm,dir=WE,label=$2$,labelstyle=right](C) 
\Loop[dist=0.99cm,dir=EA,label=$3$,labelstyle=left](B);
\Loop[dist=0.99cm,dir=WE,label=$2$,labelstyle=right](F) 
\Loop[dist=0.99cm,dir=EA,label=$3$,labelstyle=left](G);
\path 
(A) edge [bend right=70] node[] {0} (C) 
(A) edge [bend left=70] node[] {1} (C) 
(B) edge [bend right] node {3} (C)
(C) edge [bend right] node {2} (B)
(F) edge  [bend right] node {4} (G)
(E) edge  node {0} (H)
(E) edge  node {2} (I) 
(I) edge  [bend left] node {3} (F)
(H) edge  [bend right] node {4} (F)
(G) edge [bend right] node {1} (F); 
	\end{tikzpicture}}
	\captionof{figure}[\ The graphs of the configurations in $\mathcal{V}_F$ and $\mathcal{W}_F$.]{The graphs of the configurations in $\mathcal{V}_F$ (left) and $\mathcal{W}_F$ (right).} \label{fig:VFandWF}
\end{center}
\end{rema}

We now discuss the essence of Example \ref{e:BTWlikeex} and thereafter we prove our first general statement. Since $f=-u^{-1}+2$ is a sandpile polynomial we immediately see by applying the burning algorithm \eqref{eq:infiniteba} that a symbolic representation $\mathcal{R}^{(f)}$ of $X_{\tilde{f}}$ via $\xi_f$ is given by the full shift on two symbols (cf. \eqref{eq:Xh}). The subshift $\mathcal{W}^{(h)}_g$ is close to the full 2-shift in the sense that the shift $\sigma_{\mathcal{W}^{(h)}_g}$ on $\mathcal{W}^{(h)}_g$ has the same topological entropy, as is easily seen from the graph $G'^{(h)}_g$. Hence $\mathcal{W}^{(h)}_g$ is an equal entropy cover of $X_{\tilde{f}}$ via $\xi_{h}$, or equivalently, $\mathcal{V}^{(h)}_g$ is an equal entropy cover of $X_{\tilde{f}}$ via $\xi_{f}$. 

Although the description of $\mathcal{V}^{(h)}_g$ is far more complicated in the general setting, i.e., for nonsimple $h=fg \in \mathbf{P}_d$ with $d >1$, we can still apply the burning algorithm \eqref{eq:infiniteba} on $\mathcal{R}^{(h)}$ to characterise the configurations in $\mathcal{V}^{(h)}_g$: $u$ lies in $\mathcal{V}^{(h)}_g$ if and only if for any $F\Subset \mathbb{Z}^d$ there exists and $\mathbf{i} \in F$ such that
\begin{equation*}
  (g\cdot u)_\mathbf{i} \ge - \sum_{\mathbf{j\neq i}, \mathbf{j} \in F}\Delta^{h,F}_{\mathbf{ji}}.
\end{equation*}
This follows from the fact that $g\cdot \mathcal{V}^{(h)}_g = \mathcal{W}^{(h)}_g \subseteq \mathcal{R}^{(h)}$. Furthermore, the above statement about equal entropy remains true in general. 

\subsection{Equal entropy covers} In this subsection we prove the following statement and find Laurent polynomials $f,g\in \mathbf{R}_d$ which satisfy the needed assumption that $fg\in \mathbf{P}_d$.
\begin{theo} \label{t:equalentropy}
Let $d\ge 1$ and $h=f g \in \mathbf{P}_d$ such that the Laurent polynomials $f$ and $g$ have no common factor. Then $\mathcal{W}^{(h)}_g=\mathcal{R}^{(h)} \cap \ker \xi_g$ is an equal entropy cover of $X_{\tilde{f}}$. 
\end{theo}

Before proving this theorem we need some notation and a little bit of preparation. Let $F\Subset \mathbb{Z}^d$ and define the embedding $\iota_F: \mathbb{Z}^F \to \ell^1(\mathbb{Z}^d,\mathbb{Z})$ by
\begin{equation}\label{eq:iota}
\iota_F(u)_\mathbf{j}=
\begin{cases}
u_\mathbf{j}& \ \textup{if} \enspace \mathbf{j}\in F \\
0 & \textup{otherwise}.
\end{cases}
\end{equation} 
For $A\subset \mathbb{Z}^F$ we set $\iota_F(A)=\{\iota_F(a):a\in A\}$. If it is clear from the context we omit the subscript $F$ and write simply $\iota$ instead of $\iota_F$. Define further $\tilde{\iota}_F: \ell^\infty(\mathbb{Z}^d,\mathbb{Z})\to \ell^1(\mathbb{Z}^d,\mathbb{Z})$ by 
\begin{equation*}
\tilde{\iota}_F=\iota\circ\pi_F,
\end{equation*}
where $\pi_F:\ell^\infty(\mathbb{Z}^d,\mathbb{Z})\to\mathbb{Z}^F$ is the projection onto the coordinates in $F$. 

If $g$ is expansive then the shift $\alpha_{\tilde{g}}$ on $X_{\tilde{g}}$ is expansive which in turn guarantees that
\begin{equation} \label{eq:expansiveConstant}
\delta_g=\inf_{\bar{\mathbf{0}}\neq x \in X_{\tilde{g}}}\sup_{\mathbf{j}\in\mathbb{Z}^d} \pmb{|} x_\mathbf{j} \pmb{|}>0,
\end{equation}
where $\pmb{|} z \pmb{|}=\min\{|z-n|:n\in \mathbb{Z}\}, z\in \mathbb{T}=[0,1)$. Furthermore, for $0<\varepsilon < \delta_g$ the definition of $\bar{\xi_g}$ and the $\ell^1$-summability of the homoclinic point $w^{g}$ (cf. \eqref{eq:xih} and \eqref{eq:homoclinicequation2}) implies that there exists an $M=M_g\ge 1$ such that 
\begin{equation}\label{eq:xicontinuity}
Q_{M}\supset \textup{supp}(g) \qquad \text{and} \qquad
| \bar{\xi}_g(u)_\mathbf{0}-\bar{\xi}_g(v)_\mathbf{0}|< \varepsilon
\end{equation} 
for all $u,v \in \ell^\infty(\mathbb{Z}^d,\mathbb{Z})$ with $\pi_{Q_M}(u)=\pi_{Q_M}(v)$, where
\begin{equation}\label{eq:QM} 
Q_M=\{-M,\dots,M\}^d.
\end{equation}
In particular, for $u,v \in \ell^\infty(\mathbb{Z}^d,\mathbb{Z})$ and $L\ge 1$ we have that
\begin{equation}\label{eq:xicontinuity2}
\pi_{Q_{L+M}}(g\cdot u)=\pi_{Q_{L+M}}(g\cdot v)\qquad \textup{implies} \qquad\pi_{Q_{L}}(u)=\pi_{Q_{L}}(v).
\end{equation} 
Finally, we define for every $j,K\ge 1$ and $u\in \mathcal{W}^{(h)}_g$ the set $C^{(j,K)}(u) \subseteq\mathcal{W}^{(h)}_g$ by
\begin{equation}\label{eq:cylinder}
C^{(j,K)}(u)=\{ v \in \mathcal{W}_g^{(h)}: v_\mathbf{k}=u_\mathbf{k}\ \textup{for all}\ \mathbf{k}\in Q_{j+K}\setminus Q_j\}.
\end{equation} 

The following two lemmas are useful in order to determine the topological entropy $\textup{h}_{\textup{top}}(\sigma_{\mathcal{W}^{(h)}_g})$. For expansive Laurent polynomials $f,g$ we put $M_{f,g}=\max(M_g,M_f)$.
\begin{lemm}\label{l:nofmultiples}
Let $d \ge 1$ and $h=fg \in \mathbf{P_d}$ such that $f,g \in \mathbf{R}_d$ have no common factor. Let further $j\ge 1$ and $N\ge 3M_{f,g}$. If $u \in \mathcal{W}^{(h)}_g$ then any pair of distinct points $v,w \in C^{(j,N)}(u)$ satisfies that $\tilde{\iota}_{Q_j}(v-w) \notin (f)=f\cdot \mathbf{R}_d$.
\end{lemm}
\begin{proof}
We prove the statement indirectly. Assume that 
\[\tilde{\iota}_{Q_j}(v)=\tilde{\iota}_{Q_j}(w) + f\cdot m\]
for some $m\in \mathbf{R}_d$ with $\textup{supp}(f\cdot m) \subset Q_j$. Since $v,w \in \mathcal{W}^{(h)}_g\subset  \mathcal{R}^{(h)}$ we can find $v',w' \in \mathcal{V}^{(h)}_g$ such that $v-w=g\cdot(v'-w')$. By assumption $(v-w)_\mathbf{k}=0$ for all $\mathbf{k}\in Q_{j+N}\setminus Q_j$ and \eqref{eq:xicontinuity2} implies that $\textup{supp}(v'-w') \cap Q_{j+N-M} \subseteq Q_{j+M}$, where $M=M_{f,g}$, and therefore
\[f\cdot m=\tilde{\iota}_{Q_j}(g\cdot(v'-w'))=g\cdot\tilde{\iota}_{Q_{j+M}}(v'-w'). \]
Since $f$ and $g$ have no common factor and $\mathbf{R}_d$ has unique factorisation $g$ divides $m$, i.e. $m=g\cdot m'$ for some $m' \in \mathbf{R}_d$. It follows that 
\[ \pi_{Q_j}(v)=\pi_{Q_j}(w) + h\cdot m'. \]
Let $\tilde{v}$  satisfy $\pi_{Q_j}(\tilde{v})=\pi_{Q_j}(v)$ and $\pi_{\mathbb{Z}^d\setminus Q_j}(\tilde{v})=\pi_{\mathbb{Z}^d\setminus Q_j}(v_{\max})$, where $(v_{\max})_\mathbf{j}=\gamma_h-1$ for all $\mathbf{j}\in \mathbb{Z}^d$ (cf. Definition \ref{d:dompol}), and define $\tilde{w}$ similarly. Then $\tilde{v},\tilde{w} \in \mathcal{R}^{(h)}$ (cf. \eqref{eq:infiniteba}) and $\tilde{v}-\tilde{w} \in (h)=h\cdot \mathbf{R}_d$.  A contradiction to Lemma \ref{l:propxi} (2).
\end{proof}

For the next lemma we use the notion of $(F,\varepsilon')$-separated set, where $F\Subset \mathbb{Z}^d$ and $\varepsilon' >0$. A subset $Z \subset X_{\tilde{f}}$ is called {\it $(F,\varepsilon')$-separated} if for any pair of distinct points $x,y\in Z$ we can find a $\mathbf{j}\in F$ such that $\pmb{|} x_\mathbf{j}-y_\mathbf{j} \pmb{|}>\varepsilon'$. For better readability of the next lemma we put $M=\max(M_g,M_f)$. 
 
\begin{lemm}\label{l:separatedimage}
Let $d \ge 1$ and $h=fg \in \mathbf{P_d}$ such that $f$ and $g$ have no common factor. For every $j\ge 1, N\ge 3M$ and $u\in \mathcal{W}^{(h)}_g$ the map $\xi_f$ restricted to $\tilde{\iota}_{Q_j}(C^{(j,N)}(u))$ is injective. Moreover, $\xi_f(\tilde{\iota}_{Q_j}(C^{(j,N)}(u)))\subseteq X_{\tilde{f}}$ is a $(Q_{j+M},\delta_f)$-separated set (cf. \eqref{eq:cylinder}).
\end{lemm}
\begin{proof}
The first statement follows from Lemma \ref{l:nofmultiples} and the third equation in \eqref{eq:kernel}. We prove the second statement indirectly. Let $v,w \in C^{(j,N)}(u) \subseteq \mathcal{W}^{(h)}_g$ and put $x=\xi_f(\tilde{\iota}(v)),y=\xi_f(\tilde{\iota}(w)) \in X_{\tilde{f}}$. Suppose that
\begin{equation}\label{eq:E6} 
\pmb{|} x_\mathbf{n}-y_\mathbf{n} \pmb{|} <\delta_f, 
\end{equation}
for every $\mathbf{n} \in Q_{j+M}$, where $\delta_f$ is defined in \eqref{eq:expansiveConstant}. From (\ref{eq:xicontinuity}) we obtain that (\ref{eq:E6}) also holds for all $\mathbf{n}\in \mathbb{Z}^d\setminus Q_{j+M}$, since $\pi_{\mathbb{Z}^d\setminus Q_{j}}(\tilde{\iota}(v))=\pi_{\mathbb{Z}^d \setminus Q_{j}}(\tilde{\iota}(w))$, and therefore for all $\mathbf{n} \in \mathbb{Z}^d$. By expansiveness of $f$ Equation (\ref{eq:expansiveConstant}) implies that $\xi_f(\tilde{\iota}(v))=x=y=\xi(\tilde{\iota}(w))$. 
\end{proof}
\begin{proof}[Proof of Theorem \ref{t:equalentropy}]
We first show that $\xi_h(\mathcal{W}^{(h)}_g)=X_{\tilde{f}}$. Recall that $\mathcal{V}^{(h)}_g=\bar{\xi}_g(\mathcal{W}^{(h)}_g)\subseteq \ell^\infty(\mathbb{Z}^d,\mathbb{Z})$, so that  $\xi_h(\mathcal{W}^{(h)}_g)=\rho(\bar{\xi}_f(\bar{\xi}_g(\mathcal{W}^{(h)}_g)))=\xi_f(\mathcal{V}^{(h)}_g)\subseteq X_{\tilde{f}}=\ker \tilde{f}(\alpha)$. Conversely, if $x\in X_{\tilde{f}}$ choose $w \in \rho^{-1} ( x )$ (cf. \eqref{eq:rho}) with $0 \le w_\mathbf{j} <1$ for every $\mathbf{j} \in \mathbb{Z}^d$. Then $u={f} \cdot w \in \ell^\infty(\mathbb{Z}^d,\mathbb{Z})$ and $ \xi_h({g}\cdot u)=\rho(\bar{\xi}_f(u))=x$, where we have used Equation \eqref{eq:hinverse}. By Lemma \ref{l:propxi} there exists an $n \in \ell^\infty(\mathbb{Z}^d,\mathbb{Z})$ such that $v={g}\cdot u+{h}\cdot n={g}\cdot(u +{f} \cdot n) \in \mathcal{R}^{(h)}$. Therefore, $v \in \mathcal{W}^{(h)}_g$ and $\xi_h(v)=\xi_f(u)=x$. 

We have shown that $\xi_h(\mathcal{W}^{(h)}_g)=X_{\tilde{f}}$. Together with the equivariance of $\xi_h$ (cf. \eqref{eq:kernel}) we deduce that $X_{\tilde{f}}$ is a symbolic cover of $\mathcal{W}^{(h)}_g$, hence $\textup{h}_{\textup{top}}(\sigma_{\mathcal{W}^{(h)}_g}) \ge h_{\textup{top}}(\alpha_f)$. It remains to show the reverse inequality. 

For every $K\ge 1, \varepsilon > 0$ denote by $s_{K}(\varepsilon)$ the maximal cardinality of any $(Q_{K},\varepsilon)$-separated set in $X_{\tilde{f}}$ (cf. \eqref{eq:QM}). By definition of topological entropy we get
\begin{equation}
	\label{eq:Whgsize}
	\begin{aligned} 
\textup{h}_{\textup{top}}(\sigma _{\mathcal{W}^{(h)}_g})&=\lim_{L\to\infty }\frac{1}{|Q_L|} \log\, \bigl| \pi_{Q_L}(\mathcal{W}^{(h)}_g) \bigr|
\\
&=\lim_{L\to\infty }\frac{1}{|Q_L|} \sup_{u \in \pi_{Q_{L+N}\setminus Q_L}(\mathcal{W}^{(h)}_g)}\log\,\bigl|\pi _{Q_L}({\pi}^{-1}_{Q_{L+N} \setminus Q_L}(u))\bigl|
\\
&=\lim_{L\to\infty }\frac{1}{|Q_L|} \sup_{u \in \pi_{Q_{L+N}\setminus Q_L}(\mathcal{W}^{(h)}_g)}\log\,\bigl|\tilde{\iota}_{Q_L}(C^{(L,N)}(u))\bigl|
\\
&=\lim_{L\to\infty }\frac{1}{|Q_L|} \sup_{u \in \pi_{Q_{L+N}\setminus Q_L}(\mathcal{W}^{(h)}_g)}\log\,\bigl|\xi_f(\tilde{\iota}_{Q_L}(C^{(L,N)}(u) ))\bigl|
\\
&\le \lim_{L\to\infty }\frac{1}{|Q_{L+M}|} \log s_{L+M}(\delta_f) = \textup{h}_{\textup{top}}(\alpha _f),
	\end{aligned}
	\end{equation}
where we have used the Lemmas \ref{l:nofmultiples} and \ref{l:separatedimage}.
\end{proof}
We have achieved our first goal of this section, but at this point we do not know much about its usefulness. In Example \ref{e:BTWlikeex} we have found two Laurent polynomials $f,g\in \mathbf{P}_1$ such that $h=fg \in \mathbf{P}_1$. However, in that case Theorem \ref{t:equalentropy} does not provide any new insights as an explicit symbolic representation of $X_{\tilde{f}}$ is already known (cf. Theorem \ref{t:symbolicrepresentation}). The following examples illustrates the usefulness of Theorem \ref{t:equalentropy} in dimension $d=1$. At least one factor of $h$ does not lie in $\mathbf{P}_1$ and the Laurent polynomials $f$ and $g$ have no common factor. Note that the order of the factors is interchangeable.  
\begin{exas}\ 
\begin{enumerate}
\item
For $f_1=-u^{-2}-2u^{-1}+3 +u \in \mathbf{R}_1$ and lopsided $g_1=2 -u \in \mathbf{P}_1$ (cf. Definition \ref{d:dompol}) we get
\begin{equation*}
h_1=f_1g_1= -2u^{-2}-3u^{-1}+8-u-u^2 \in \mathbf{P}_1.
 \end{equation*}
\item
The Laurent polynomials $f_2=-u^{-1}+3 +u$, $g_2=u^{-1}+3 -u \in \mathbf{R}_1$ are both lopsided and
\begin{equation*}
h_2=f_2g_3= -u^{-2}+11- u^2 \in \mathbf{P}_1.
\end{equation*}
\item
Neither $f_3=-u^{-2}-2u^{-1}+2-u+u^2$ nor $g_3=1-u-u^2\in \mathbf{R}_1$ is lopsided and
\begin{equation*}
h_3=f_3g_3= -u^{-2}-u^{-1}+5-u- u^4 \in \mathbf{P}_1.
\end{equation*}
\end{enumerate}
\end{exas}
In dimension $d>1$ the situation seems to be more obscure. The following lemma shows that each lopsided Laurent $f$ polynomial with dominant coefficient $\gamma_f=f_{\mathbf{0}}$ and purely positive coefficients has a cofactor $g$ in $\mathbf{P}_d$ such that the product $h=fg$ is a sandpile polynomial. We remark that we are not aware of any other non-trivial classes of polynomials in $\mathbf{R}_d$ with $d>1$ which fulfil this property.

\begin{lemm}\label{l:positivepol}
Fix $d \ge 1$. Let $g$ be in $\mathbf{P}_d$ and define the Laurent polynomial $f$ by $f_\mathbf{n}=|g_\mathbf{n}|$ for all $\mathbf{n} \in \mathbb{Z}^d$. Then $h=f g \in \mathbf{P}_d$. Moreover, $f$ and $g$ have no common factor.
\end{lemm}
\begin{proof}
Recall that $g\in \mathbf{P}_d$ if and only if $\sum_{\mathbf{j} \in \mathbb{Z}^d} g_\mathbf{j} >0$ and $g_\mathbf{i}\le 0$ for all $ \mathbf{0} \neq \mathbf{i} \in \mathbb{Z}^d$. In particular $g_\mathbf{0}>0$, hence $f_\mathbf{0}=g_\mathbf{0}$ and $f_\mathbf{i}=-g_\mathbf{i}$ for all $\mathbf{i}\neq \mathbf{0}$. For $\mathbf{j}\neq \mathbf{0}$ we have 
\begin{equation}\label{eq:Pd1}
h_\mathbf{j}=\sum_{\mathbf{k}\in\mathbb{Z}^d}f_\mathbf{k}g_\mathbf{j-k}=f_\mathbf{0}g_\mathbf{j}+f_\mathbf{j}g_\mathbf{0}+\sum_{\mathbf{k} \in \mathbb{Z}^d \setminus \{\mathbf{0},\mathbf{j}\}} f_\mathbf{k}g_\mathbf{j-k} \le 0, 
\end{equation}
and since $g \in \mathbf{P}_d$
\begin{equation}\label{eq:Pd2}
\sum_{\mathbf{j} \in \mathbb{Z}^d} h_\mathbf{j}=h(\bar{1})=f(\bar{1})g(\bar{1})>0,
\end{equation}
where $\bar{1}=(1,\dots,1)\in \mathbb{Z}^d$. From \eqref{eq:Pd1} and \eqref{eq:Pd2} we conclude that $h\in \mathbf{P}_d$. 

For the last assertion it suffices to prove that $f$ and $g$ have no common root. Suppose there exists an $\mathbf{r} \in \mathbb{Z}^d$ such that $f(\mathbf{r})=g(\mathbf{r})=0$. Then $2g_{\mathbf{0}}=f(\mathbf{r})+g(\mathbf{r})=0$ and therefore $g_{\mathbf{0}} = 0$, a contradiction to $g \in \mathbf{P}_d$.
\end{proof}
\begin{defi}\label{d:positivepol}
We denote the class of Laurent polynomials defined in Lemma \ref{l:positivepol} by $\mathbf{P}_d^+$.
\end{defi} 

\subsection{Symbolic representations}
The remainder of this section is devoted to our final goal to specify conditions on $f$ and $g\in \mathbf{R}_d$ with $h=fg \in \mathbf{P}_d$ such that the equal entropy cover $\mathcal{W}^{(h)}_g=\mathcal{R}^{(h)} \cap \ker \xi_g$ is in fact a symbolic representation. In other words, we want to find a unique measure of maximal entropy on $\mathcal{W}^{(h)}_g$ for which the covering map $\xi_h:\mathcal{W}^{(h)}_g \to X_{\tilde{f}}$ is almost one-to-one. For this purpose the abelian structure in finite volume, which we discussed at the beginning of this section, will become very useful. Before we get into more details we repeat some basics from the previous sections for convenience.

Let $g \in \mathbf{R}_d$, $F \Subset \mathbb{Z}^d$ and define $\Delta^g=\Delta^{g,F}$ as in \eqref{eq:topplingmatrix} by
\begin{displaymath}
\Delta^g_{\mathbf{i}\mathbf{j}}=g_\mathbf{i-j}
\end{displaymath} 
for every ${\mathbf{i},\mathbf{j}} \in F$. 
Moreover, if $g \in \mathbf{P}_d$ then the matrix $\Delta^g$ fulfils the properties

\begin{equation*}
\begin{aligned} 
  \textup{(P1)}\ \ & \Delta^g_{\mathbf{i}\mathbf{j}} \le 0\ \textup{for all}\ \mathbf{i}\neq \mathbf{j} \in F,
  \\
  \textup{(P2)}\ \ & \sum_{\mathbf{j} \in F} \Delta^g_{\mathbf{i}\mathbf{j}} > 0\ \textup{for all}\ \mathbf{i} \in F.
  \end{aligned}
\end{equation*}
\begin{rema}\label{r:asmcondition}
Any matrix $\Delta$ satisfying the properties (P1) and (P2) is called toppling matrix and determines an ASM on $F$ with recurrence class $\mathcal{R}_F$. (cf. \cite{Speer}). 
\end{rema}
As seen above $\mathcal{W}^{(h)}_g=\mathcal{R}^{(h)} \cap \ker \xi_g$ is an equal entropy cover of $X_{\tilde{f}}$ via $\xi_h$. Unfortunately, the restriction $\pi_F(\mathcal{W}^{(h)}_g)$ of $\mathcal{W}^{(h)}_g$ to $F\Subset \mathbb{Z}^d$ does not provide an obvious group structure like the one known on $\pi_F(\mathcal{R}^{(h)})=\mathcal{R}_F^{(h)}$. In that case the abelian group $\mathcal{R}^{(h)}_F$ carries the Haar measure $\mu_F$. This property is essential for finding the unique measure $\mu=\lim_{F\uparrow \mathbb{Z}^d} \mu_F$ of maximal entropy on $\mathcal{R}^{(h)}$ for which the map $\xi_h:\mathcal{R}^{(h)}\to X_{\tilde{h}}$ is almost one-to-one (cf. \cite{4Redig2},\cite{4Schmidt1}). To circumvent this problem we introduce a new toppling  matrix $\Delta'=\Delta^{g} \Delta^{f}$ of dimension $|F|\times |F|$, which replaces the role of  $\Delta^{h}$ in the standard case with $h\in \mathbf{P}_d$. The advantage of this approach is that the multiples of $\Delta^{g}$ form a subgroup $\mathcal{W}_F$ of the abelian group $\mathcal{R}'_F$ determined by $\Delta'$, as we shall see further on. This group property brings us back to well-known territory and we can apply exactly the same tools as in \cite{4Redig2} and \cite{4Schmidt1} we already implicitly used to describe Theorem \ref{t:symbolicrepresentation}. There are two requirements we need to check in order to use these tools. First we need to ensure that $\Delta'$ really determines an ASM in finite volume so that we can consider $\mathcal{W}_F$ as abelian subgroup and second that any accumulation point of the sequence of the Haar measures $\{\nu_F\}_{F\uparrow \mathbb{Z}^d}$ on $\mathcal{W}_F$ concentrates on $\mathcal{W}^{(h)}_g$ and that it is a measure of maximal entropy. 

We start by investigating the first requirement in finite volume with expansive $f$ and $g$. The expansiveness condition is needed to ensure that the ASM determined by $\Delta'$ is dissipative when we perform the transition to infinite volume (cf. Definition \ref{d:dissipative} and the discussion thereafter). We easily see that the matrix $\Delta'=\Delta^g\Delta^f$ is a toppling matrix for $f$ and $g$ in Example \ref{e:BTWlikeex} (1) and (2) but not in (3). This means that we cannot deduce from $h=fg$ being a sandpile polynomial that $\Delta'$ determines an ASM in finite volume. For the examples in higher dimension $d\ge 1$ we recall that each $g \in \mathbf{P}_d$ defines a polynomial $f\in \mathbf{P}_d^+$ such that $h=fg \in \mathbf{P}_d$ (cf. Lemma \ref{l:positivepol} and Definition \ref{d:positivepol}). In this case we call the polynomial $f$ {\it associated to $g$}. 
\begin{lemm}
Let $d\ge 1$ and fix $F \Subset \mathbb{Z}^d$. If $g \in \mathbf{P}_d$ with associated Laurent polynomial $f \in \mathbf{P}_d^+$, then the ($|F|\times |F|$)-matrix $\Delta'=\Delta^{g} \Delta^{f}$ determines an ASM on $F$.
\end{lemm}
\begin{proof}
According to Remark \ref{r:asmcondition} we only need to check the properties (P1) and (P2) for $\Delta'$. By assumption $f_\mathbf{i} \ge 0$ for all $\mathbf{i}\in \mathbb{Z}^d$, $f_\mathbf{0}=g_\mathbf{0}$ and $f_\mathbf{i}=-g_\mathbf{i}$ for all $\mathbf{i}\neq \mathbf{0}$. If $\mathbf{i,j} \in F, \mathbf{i}\neq \mathbf{j}$ we obtain
\begin{eqnarray} \label{eq:Delta'equality}
\Delta'_{\mathbf{ij}}&=&\sum_{\mathbf{k}\in F}\Delta^{g}_{\mathbf{ik}}\Delta^{f}_{\mathbf{kj}}=\sum_{\mathbf{k}\in F}g_{\mathbf{i-k}} f_{\mathbf{k-j}}=\sum_{\mathbf{k}\in F-\mathbf{j}}f_{\mathbf{k}} g_{\mathbf{(i-j)-k}} \\ \label{eq:P1}
&=& f_\mathbf{0}g_\mathbf{i-j}+f_\mathbf{i-j}g_\mathbf{0}+\sum_{\mathbf{k} \in (F-\mathbf{j})\setminus \{\mathbf{0},\mathbf{i-j}\}} f_\mathbf{k}g_\mathbf{(i-j)-k} \le 0 
\end{eqnarray} 
Furthermore, since $f_{\mathbf{k}} g_{\mathbf{(i-j)-k}} \le 0$ for $\mathbf{k} \in \mathbb{Z}^d \setminus F -\mathbf{j}$ we obtain from (\ref{eq:Delta'equality}) for all $\mathbf{i,j} \in F$ 
\begin{eqnarray*}
\Delta'_{\mathbf{ij}}&=&\sum_{\mathbf{k}\in \mathbb{Z}^d}f_{\mathbf{k}}\cdot g_{\mathbf{(i-j)-k}}-\sum_{\mathbf{k}\in \mathbb{Z}^d\setminus F-\mathbf{j}}f_{\mathbf{k}}\cdot g_{\mathbf{(i-j)-k}}
\\
&=&h_{\mathbf{i-j}}-\sum_{\mathbf{k}\in \mathbb{Z}^d\setminus F-\mathbf{j}}f_{\mathbf{k}}\cdot g_{\mathbf{(j-i)-k}} \ge h_\mathbf{i-j}=\Delta^{h}_{\mathbf{ij}},
\end{eqnarray*}
where equality holds if supp($f$)$\subseteq F -\mathbf{j}$
. Furthermore,
\begin{equation}\label{eq:P2}
\sum_{\mathbf{j}\in F} \Delta'_{\mathbf{ij}}\ge \sum_{\mathbf{j}\in F} h_\mathbf{i-j} > 0,
\end{equation}
as $h\in \mathbf{P}_d$. By \eqref{eq:P1} and \eqref{eq:P2} the matrix $\Delta'$ is a toppling matrix.
\end{proof}
For the remainder of the section fix $d \ge 1$ and suppose that for $f,g \in \mathbf{R}_d$ the ($|F|\times |F|$)-matrix $\Delta'=\Delta^g\Delta^f$ defines a toppling matrix for any $F\Subset \mathbb{Z}^d$. Since the product $h=fg$ then also lies in $\mathbf{P}_d$ this is the proper setting for our purposes. Under the above assumption on $f$ and $g$ we now investigate the relation between the recurrence classes $\mathcal{R}'_F$ and $\mathcal{R}_F^{(h)}$ determined by $\Delta'$ and $\Delta^{h}$, respectively.

\begin{prop}\label{p:Delta'} Let $f$ and $g \in \mathbf{R}_d$ have no common factor such that for each $F \Subset \mathbb{Z}^d$ the ($|F|\times |F|$)-matrix $\Delta'=\Delta^{g} \Delta^{f}$ determines an ASM on $F$ with recurrence class $\mathcal{R}'_F$.
Put $F^\circ=\{ \mathbf{i} \in F: \mathbf{i} +\textup{supp}(f) \subset F\}$, then 
\[
\pi_{F^\circ}(\mathcal{R'}_F)=\mathcal{R}^{(h)}_{F^\circ},
\] 
where $\pi_{F^\circ}$ is the projection onto the coordinates in $F^\circ$. 
\end{prop}
\begin{proof}
Recall from the burning algorithm \eqref{eq:finiteba} that $v \in \mathcal{R}'_F$ if and only if for all $G \subseteq F$ exists an $\mathbf{i}\in G$ such that
\begin{displaymath}
\Delta'_{\mathbf{ii}}> v_\mathbf{i}\ge -\sum_{\mathbf{j}\in G,\mathbf{j\neq i}}\Delta'_{\mathbf{ij}}.
\end{displaymath}
In particular we also have that for every $G^\circ\subseteq F^\circ$ there exists an $\mathbf{i}\in G^\circ$ such that 
\begin{displaymath}
\Delta^{h}_\mathbf{ii}=\Delta'_{\mathbf{ii}}> v_\mathbf{i}\ge -\sum_{\mathbf{j}\in G^\circ, \mathbf{j\neq i}}\Delta'_{\mathbf{ij}}=-\sum_{\mathbf{j}\in G^\circ, \mathbf{j\neq i}}\Delta^{h}_{\mathbf{ij}},
\end{displaymath}
which proves that $\pi_{F^\circ}(v)\in \mathcal{R}^{(h)}_{F^\circ}$. Hence $\pi_{F^\circ}(\mathcal{R}'_F)\subseteq \mathcal{R}^{(h)}_{F^\circ}$. Conversely, for $v \in \mathcal{R}^{(h)}_{F^\circ}$ define 
\begin{equation*}
w_{\mathbf{i}}=
\begin{cases}
v_{\mathbf{i}}& \textup{if}\enspace \mathbf{i}\in F^\circ \\ \Delta'_{\mathbf{ii}}-1& \textup{if} \enspace \mathbf{i}\in F\setminus F^\circ.
\end{cases}
\end{equation*}
Then $w \in \mathcal{R}'_F$ and $\pi_{F^\circ}(w)=v$. It follows that $\mathcal{R}_{F^\circ}^{(h)} \subseteq \pi_{F^\circ}(\mathcal{R}'_F)$.
\end{proof}
For any $F\Subset \mathbb{Z}^d, h \in \mathbf{P}_d$ and $v\in \mathcal{R}^{(h)}_F$ we have that $v_{\mathbf{i}} < \Delta^{h}_\mathbf{ii}=\gamma_h$ for all $\mathbf{i}\in F$ (cf. Definition \ref{d:dompol}). As one would expect a similar statement is true for any $v \in \mathcal{R}'_{F}$.
\begin{coro}\label{c:beta}
Let $f$ and $g$ be as in Proposition \ref{p:Delta'} and set 
\begin{equation}\label{eq:gamma'}
\gamma'= \sum_{\mathbf{k} \in G^+} f_{\mathbf{k}}g_{-\mathbf{k}}\ge \gamma_h,
\end{equation}
where $G^+=\{\mathbf{k}\in \mathbb{Z}^d:f_{\mathbf{k}}g_{-\mathbf{k}} >0 \}$.  Then $\mathcal{R}'_F \subseteq \{0,\dots,\gamma'-1\}^F$ for any $F \Subset \mathbb{Z}^d$. 
\end{coro}
\begin{proof}
Since $h=fg \in \mathbf{P}_d$ the dominant term $\gamma_h$ of $h$ equals $h_{\mathbf{0}}=\sum_{\mathbf{k} \in \mathbb{Z}^d} f_{\mathbf{k}}g_{-\mathbf{k}}>0$. Hence $G^+$ is nonempty and we obtain for any $F \Subset \mathbb{Z}^d$ and $\mathbf{i}\in F$ that $v_{\mathbf{i}}<\Delta'^F_{\mathbf{ii}}=\sum_{\mathbf{k} \in F} g_{\mathbf{i-k}}f_{\mathbf{k-i}} \le \gamma'$. 
\end{proof}

As our next step we want to show that for given $F \Subset \mathbb{Z}^d$ the set \begin{equation}\label{eq:R'g}
\mathcal{W}_F=\{\Delta^{g}v \in \mathcal{R}'_F:v \in \mathbb{Z}^F \}\subseteq \mathcal{R}'_F
\end{equation} 
of multiples of $\Delta^g$ in $\mathcal{R}'_F$ form a subgroup of the abelian group $\mathcal{R}'_F$. Before proving this statement we take a closer look into the group $\mathcal{R}'_F$. 

Following the same steps of Section 2.3. in \cite{Redig1}, by simply replacing the toppling matrix therein with $\Delta'$, we see that the neutral element ${e}_{\mathcal{R}'_F}$ of the abelian group $\mathcal{R}'_F$ is given by
\begin{equation}\label{eq:neutral}
{e}_{\mathcal{R}'_F}=\Delta'm
\end{equation}
for some $m \in \mathbb{Z}^F$. Moreover, the addition of the elements $u,v \in \mathcal{R}'_F$ is realised as
\begin{equation}\label{eq:stableAfterToppling}
u\oplus v=u + v-\Delta'n,
\end{equation}
where $n_\mathbf{i}$ gives the numbers of topplings at site $\mathbf{i} \in F$ needed to stabilise $u + v$.
\begin{rema}\label{r:systemofreps}
If $w \in \mathbb{Z}^F$ then there exists a unique $m=m_w\in \mathbb{Z}^F$ such that $w'=w-\Delta'm \in \mathcal{R}'_F$. In particular, two distinct elements $v,w$ in $\mathcal{R}'_F$ cannot differ by a multiple of $\Delta'$. 
\end{rema}
\begin{prop}\label{l:subgrooupR'g}
For any $F\Subset \mathbb{Z}^d$ the set $\mathcal{W}_F$ is a subgroup of $\mathcal{R}'_F$.
\end{prop}
\begin{proof}
Fix $F \Subset \mathbb{Z}^d$. If $m\in \mathbb{Z}^F$ then $\Delta'm$ is a multiple of $\Delta^{g}$, since $\Delta'm=\Delta^{g}(\Delta^{f}m)$. By \eqref{eq:neutral} it follows that ${e}_{\mathcal{R}'_F} \in \mathcal{W}_F$. If $u=\Delta^{g} u',v=\Delta^{g} v' \in \mathcal{W}_F$ then $u\oplus v=u+v-\Delta'm$ for some $m \in \mathbb{Z}^F$ (cf. \eqref{eq:stableAfterToppling}). Hence $u\oplus v=\Delta^{g}(u'+v'-m')\in \mathcal{W}_F$. If $u\in \mathcal{W}_F$ then $u^{-1}= e_{\mathcal{R}'_F} \ominus u\in \mathcal{W}_F$ by the same reasoning as before.
\end{proof}
We now draw our attention to our last goal of finding a measure which concentrates on $\mathcal{W}^{(h)}_g$ and has maximal entropy. For the first statement we will need the following lemma. 
\begin{lemm}\label{l:cofactorBound}
There exists a constant $\beta$ such that for any $F \Subset \mathbb{Z}^d$ and $v=\Delta^{g}v' \in \mathcal{W}_F$ with $v' \in \mathbb{Z}^F$ we obtain $0 \le v'_\mathbf{i} < \beta$ for all $\mathbf{i} \in F$.
\end{lemm}
\begin{proof}
Any matrix satisfying the properties (P1) and (P2) of a toppling matrix is a special case of a so-called {\it strictly diagonally dominant matrix}. It is well known (cf. \cite{Varah}) that a matrix $A$ of this kind has a bounded inverse, namely 
\[
\|A^{-1}\|_\infty\le \max_{\mathbf{i} \in F} \frac{1}{|A_\mathbf{ii}|-\sum_{\mathbf{j\neq i}}|A_{\mathbf{ij}}|}
\]
 and therefore $\|\Delta'^{-1}\|_\infty\le 1$ for any $F \Subset \mathbb{Z}^d$. In particular, the matrix $\Delta'$ and its factors $\Delta^f$ and $\Delta^g$ are invertible. If $\Delta^g v' =v$ then 
 \[
 \|v'\|_{\max}=\|(\Delta^f\Delta'^{-1})v\|_{\max}\le \|\Delta^f\|_\infty\|\Delta'^{-1}\|_\infty\|v\|_{\max},
 \] 
 where $\|.\|_{\max}$ is the maximum norm on $\mathbb{Z}^F$ and $\|.\|_\infty$ its induced matrix norm. By definition of $\Delta^f$ and Corollary \ref{c:beta} the first and the third norm is bounded by $\|f\|_1=\sum_{\mathbf{k}\in \mathbb{Z}^d} |f_{\mathbf{k}}|=:c_1$ and $\gamma'-1$ (cf. \eqref{eq:gamma'}), respectively. The statement now follows by putting $\beta=c_1 \gamma'$.
\end{proof}
\begin{rema}
For $g\in \mathbf{P}_d, d\ge 1$ and associated $f\in \mathbf{P}_d^+$ one can show that $0\le v'_{\mathbf{i}} < \gamma_h$, where $h=fg \in \mathbf{P}_d$ (cf. Definition \ref{d:dompol}).
\end{rema}
As an abelian group the set $\mathcal{W}_F$ of multiples of $\Delta^g$ in $\mathcal{R}'_F$ carries a unique addition-invariant measure $\mu^{(g)}_F$. Since $\Lambda_{\gamma'}=\{0,\dots,\gamma'-1\}^{\mathbb{Z}^d}$ is compact, every sequence $(\mu^{(g)}_{k})_{k \ge 1}$ of Haar measures on $\mathcal{W}_{F_k}\subset \mathcal{R}'_{F_k}\subset \Lambda_{\gamma'}$ along $F_k\uparrow \mathbb{Z}^d$ contains a convergent subsequence. Thus any accumulation point of $(\mu^{(g)}_{k})_{k}$ is a measure $\mu^{(g)}$ on $\Lambda_{\gamma'}$.

\begin{prop}\label{p:concentration}
Let $\mu^{(g)}$ be an accumulation point of $(\mu^{(g)}_{k})_{k \ge 1}$. Then $\mu^{(g)}$ concentrates on $\mathcal{W}^{(h)}_g=\mathcal{R}^{(h)}\cap \ker\xi_g$.
\end{prop}
\begin{proof}
Let $(k_j:j\ge1)$ be a subsequence such that $\mu^{(g)}=\lim_{j\to \infty} \mu^{(g)}_{k_j}$. We need to show that $\mu^{(g)}(\mathcal{W}^{(h)}_g)=1$. First note that $\mu^{(g)}(\mathcal{R}^{(h)})=1$. To see this recall from \eqref{eq:infiniteba} that any $v \in \mathcal{R}^{(h)}$ satisfies $v\in \mathcal{R}^{(h)}_F$ for every $F\Subset \mathbb{Z}^d$.

Fix $F\Subset \mathbb{Z}^d$. Since $\mu_{k_j}^{(g)}$ is concentrated on $\mathcal{W}_{F_{k_j}}\subset \mathcal{R}'_{F_{k_j}}$ and $\pi_{F^\circ_{k_j}}(\mathcal{R}'_{F_{k_j}})=\mathcal{R}^{(h)}_{F^\circ_{k_j}}$ by Proposition \ref{p:Delta'}, it follows that 
\[
\mu_{k_j}^{(g)}(\{v \in \mathcal{W}_{F_{k_j}}:\pi_F(v)\in \mathcal{R}^{(h)}_F\})=1
\]
 for all $j\ge J$, where $J$ is the smallest integer $L$ such that $F_{k_L}^\circ \supseteq F$. Next we define for any $N\ge 1$ and $v \in \mathcal{R}^{(h)}$ the cylinder set $C(v,N)=\{w \in \mathcal{R}^{(h)}:\pi_{Q_N}(w)=\pi_{Q_N}(v)\}$ and consider the set
\begin{displaymath}
A=\{v\in \mathcal{R}^{(h)}:\mu^{(g)}(C(v,N))>0\ \textup{for all}\ N \ge 1\}.
\end{displaymath}
We show the statement that $\mu^{(g)}(A)=1$ in a separate Lemma. 
\begin{lemm}
The set $A$ has full measure with respect to $\mu^{(g)}$.
\end{lemm} 
\begin{proof}
Set $B=\mathcal{R}^{(h)} \setminus A$. We prove $\mu^{(g)}(B)=0$. If $w \in B$ then we can find a minimal $M_w \ge 1$ such that $\mu^{(g)}(C(w,M_w))=0$. Define a relation $\sim$ on $B$ by 
\begin{displaymath}
v \sim w\quad \textup{if and only if}\quad v \in C(w,M_w). 
\end{displaymath}
The relation $\sim$ is an equivalence relation. Reflexivity is obvious. For the symmetry let $v \sim w$. Then $v \in C(w,M_w)\subseteq B$ and $\pi_{Q_{M_w}}(v)=\pi_{Q_{M_w}}(w)$, which in turn implies that $w \in C(v,M_w)=C(w,M_w)$. The symmetry follows if $M_v=M_w$. Since $\mu^{(g)}(C(v,M_w))=\mu^{(g)}(C(w,M_w))=0$ we get $M_v \le M_w$ which implies $C(w,M_v)=C(v,M_v)$. At the same time $\mu^{(g)}(C(w,M_v))=\mu^{(g)}(C(v,M_v))=0$ and therefore $M_v \ge M_w$ by minimality of $M_w$. For the transitivity note that the proof of the symmetry shows that $v \sim w$ if and only if $C(v,M_v)=C(w,M_w)$. Hence, if $u\sim v$ and $v \sim w$ then $C(u,M_u)=C(v,M_v)=C(w,M_w)$ and therefore $u \sim w$. Denote the equivalence class containing $w$ by $[w]=\{v \in B: v \sim w\}$. Then 
\begin{equation}\label{eq:setB}
B =\bigcup_{[w] \in B/\sim} C(w,M_w),
\end{equation}
where $B/\!\! \sim$ is the set of equivalence classes in $B$. Denote by $\mathcal{C}$ the collection of all finite cylinder sets in $\mathcal{R}^{(h)}$ and note that $\mathcal{C}$ is countable. Since the map $\psi:[w] \mapsto C(w,M_w)$ is a bijection from $B/\!\! \sim$ to $\psi(B/\!\! \sim)\subseteq \mathcal{C}$, Equation \eqref{eq:setB} shows that $B$ is a countable union of sets of measure zero. It follows that $\mu^{(g)}(B)=0$. 
\end{proof}
We complete our proof by showing $A \subseteq \mathcal{W}^{(h)}_g$. If $v \in A$ then there exists for any $N\ge 1$ a minimal $L=L(N)$ such that $\pi_{Q_N}(v)\in \pi_{Q_N}(\mathcal{W}_{F_{k_j}})$ for all $j\ge L$ (cf. \eqref{eq:QM}), otherwise $\mu^{(g)}(C(v,N))=0$. Together with \eqref{eq:R'g} this allows us to find for any $j \ge L$ an element $v^{(j)} \in \mathbb{Z}^{F_{k_j}}$ such that 
\begin{equation}\label{eq:finitev}
\pi_{Q_N}(v)=\pi_{Q_N}(\Delta^{g,F_{k_j}} v^{(j)}).
\end{equation} By Lemma \ref{l:cofactorBound} we can assume that $v_\mathbf{i}^{(j)}< \beta$ for every $\mathbf{i}\in Q_N$ and $j\ge L$. 

Fix $N > 3M$, where $M=M_g$ (cf. \eqref{eq:xicontinuity}). Since $g$ is expansive we deduce from \eqref{eq:xicontinuity2} that 
\begin{equation}\label{eq:cofactor}
\pi_{Q_{N-M}}(v^{(j)})=\pi_{Q_{N-M}}(v^{(L)})
\end{equation} 
for all $j \ge L=L(N)$. In particular, $v^{(L(N+1))}_\mathbf{i}=v^{(L(N))}_\mathbf{i}$ for all $\mathbf{i} \in Q_{N-M}$ since $L(N+1)\ge L(N)$ for every $N > 3M$. For $K\ge N$ we set 
\[D(v,K)=\{w \in \{0,\dots,\beta-1\}^{\mathbb{Z}^d}: \pi_{Q_{K-M}}(w)=\pi_{Q_{K-M}}(v^{(L(K))})\}.\] 
Since $D(v,K)$ is compact and $D(v,K+1)\subseteq D(v,K)$ for every $K\ge N$ the intersection $\bigcap_{K \ge N} D(v,K)\subseteq \{0,\dots,\beta-1\}^{\mathbb{Z}^d}$ is nonempty. 
\begin{lemm}
Let $v' \in \bigcap_{K \ge N} D(v,K)$. Then $v=g\cdot v'$.
\end{lemm}
\begin{proof}
We prove the statement by showing that $\pi_{Q_n}(g\cdot v')=\pi_{Q_n}(v)$ for every $n\ge N$. For given $n$ we can find integers $m\ge n, l\ge m+M$ such that $F_{k_{L(l)}}\supseteq Q_m \supseteq Q_n + \textup{supp}(\tilde{g})$. Hence 
\begin{displaymath}
\pi_{Q_n}(g\cdot v')=\pi_{Q_n}(\Delta^{g,F_{Q_m}}\pi_{Q_m}(v'))=\pi_{Q_n}(g\cdot \iota_{Q_m}(v')).
\end{displaymath}
By \eqref{eq:cofactor} and since $v' \in D(v,m+M)$ we obtain 
\begin{displaymath}
\pi_{Q_n}(g\cdot \iota_{Q_m}(v'))=\pi_{Q_n}(g\cdot \iota_{Q_m}(v^{(L(m+M))}))=\pi_{Q_n}(g\cdot \iota_{Q_m}(v^{(L(l))})).
\end{displaymath}
And finally, by \eqref{eq:finitev} we get
\begin{displaymath}
\pi_{Q_n}(g\cdot \iota_{Q_{m}}(v^{(L(l))}))=\pi_{Q_n}(g\cdot \iota_{F_{k_{L(l)}}}(v^{(L(l))}))=\pi_{Q_n}(v).
\end{displaymath}
\end{proof}
We are now able to finish the proof of Proposition \ref{p:concentration}. Since $v' \in \ell^\infty(\mathbb{Z}^d,\mathbb{Z})$ the statement of the lemma shows that $v \in \ker \xi_g$ and hence $v\in \mathcal{W}^{(h)}_g$. Therefore $A \subseteq \mathcal{W}^{(h)}_g$ which implies that $\mu^{(g)}(\mathcal{W}^{(h)}_g)=1$. This completes the proof.
\end{proof}
\begin{rema}\label{r:uniqueness}
Proposition \ref{p:concentration} allows us to apply the same proofs as in \cite{4Redig2} (cf. Remark \ref{r:analogy}) and to show that every limit $\lim_{F\uparrow \mathbb{Z}^d}\mu^{(g)}_F$ converges to the unique accumulation point $\mu^{(g)}$ concentrated on $\mathcal{W}^{(h)}_g$ with the property that
\begin{equation*}
\mu^{(g)} \left( \left\{ u \in \mathcal{W}_g^{(h)}: \left|\xi_h^{-1}(\{\xi_h(u)\})\cap \mathcal{W}_g^{(h)}\right|=1 \right\} \right)=1.
\end{equation*}
\end{rema}

Our final step is to show that $\mu^{(g)}$ has maximal entropy. 
\begin{theo}\label{t:entropymeasure}
The measure theoretic entropy $\textup{h}_{\mu^{(g)}}(\sigma)$ of $\sigma=\sigma_{\mathcal{W}^{(h)}_g}$ coincides with its topological entropy $\textup{h}_{\textup{top}}(\sigma)$.
\end{theo}
Before proving Theorem \ref{t:entropymeasure} we need some preparation. Recall from \eqref{eq:xicontinuity} and \eqref{eq:QM} the definitions of $Q_M$ and $M_h$ for expansive $h$ and put for the next lemma $M= \max(M_g,M_f)$.

\begin{lemm} \label{l:estimation}
Fix $j\ge 1$ and $L> 6M$. For any $u \in \pi_{Q_{j+L}\setminus Q_j}(\mathcal{W}^{(h)}_g)$ we obtain that
\begin{displaymath}
\bigl|\pi _{Q_j}({\pi}^{-1}_{Q_{j+L+M} \setminus Q_j}(u))\bigl| \le  \bigl| \mathcal{W}_{Q_{j+L}} \bigr|.
\end{displaymath}
\end{lemm}
\begin{proof}
Set $\Delta'=\Delta'^{Q_{j+L}}$ and recall that $\Delta'=\Delta^{g}\Delta^{f}$. If $F \Subset \mathbb{Z}^d$, $w\in \ell^\infty(\mathbb{Z}^d,\mathbb{Z})$ we put $w_F=\pi_F(w) \in \mathbb{Z}^F$ and for notational convenience we write $\bar{Q}=Q_{j}$ and $\bar{Q}_K=Q_{j+K}$ for every $K\ge 1$. To prove the statement we first construct an injective mapping $\phi: \pi_{\bar{Q}_L}({\pi}^{-1}_{\bar{Q}_{L+M} \setminus \bar{Q}}(u)) \to \mathcal{W}_{\bar{Q}_L}$. 

Let $v=g\cdot v'\in {\pi}^{-1}_{\bar{Q}_{L+M} \setminus \bar{Q}}(u) \subset \mathcal{W}^{(h)}_g$ with $v' \in \mathcal{V}^{(h)}_g$ (cf. \eqref{eq:Wgh}--\eqref{eq:Vgh}) and put 
\begin{equation*}
\hat{v}=\Delta^{g}v'_{\bar{Q}_L} \in \mathbb{Z}^{\bar{Q}_L}.
\end{equation*} 
By Remark \ref{r:systemofreps} we can find a unique element $m_v \in \mathbb{Z}^{\bar{Q}_L}$ such that
\begin{displaymath}
\tilde{v}=\hat{v}-\Delta'm_v=\Delta^{g}(v'_{\bar{Q}_L}-\Delta^{f}m_v) \in \mathcal{W}_{\bar{Q}_L}.
\end{displaymath} 
By \eqref{eq:xicontinuity2} the map $\phi: \pi_{\bar{Q}_L}({\pi}^{-1}_{\bar{Q}_{L+M} \setminus \bar{Q}}(u)) \to \mathcal{W}_{\bar{Q}_L}$, given by $\phi(v_{\bar{Q}_L})=\tilde{v}$, is well-defined. We prove the injectivity of $\phi$ by contradiction.  

Let $v_{\bar{Q}_L}\neq w_{\bar{Q}_L} \in  \pi_{\bar{Q}_L}({\pi}^{-1}_{\bar{Q}_{L+M} \setminus \bar{Q}}(u))$ and assume $\tilde{v}=\tilde{w}$. Then
\begin{eqnarray*}
\hat{v}-\Delta'm_v &=& \hat{w}-\Delta'm_w \qquad \Leftrightarrow \\
\Delta^{g}(v'_{\bar{Q}_L}-w'_{\bar{Q}_L})&=&\Delta^{g}(\Delta^{f}(m_v-m_w)).
\end{eqnarray*}
Therefore
\begin{equation}\label{eq:fmultiple1}
v'_{\bar{Q}_L}-w'_{\bar{Q}_L}=\Delta^{f}(m_v-m_w)=\pi_{\bar{Q}_L}(f \cdot \iota_{\bar{Q}_L}(m_v-m_w)),
\end{equation}
where $\iota$ is the embedding defined in \eqref{eq:iota}. As $g$ is expansive and $\textup{supp}(v-w)\cap \bar{Q}_L\subseteq \bar{Q}$ we deduce from \eqref{eq:xicontinuity2} that 
\begin{equation}\label{eq:support1}
\textup{supp}(v'-w')\cap \bar{Q}_{L-M}\subseteq \bar{Q}_M
\end{equation}
and analogously by expansiveness of $f$ we get
\begin{equation}\label{eq:support2}
\textup{supp}(m_v-m_w)\cap \bar{Q}_{L-2M}\subseteq \bar{Q}_{2M}.
\end{equation}
Carefully examining \eqref{eq:fmultiple1}--\eqref{eq:support2} yields
\begin{equation*}
\tilde{\iota}_{\bar{Q}}(v-w)=g\cdot \tilde{\iota}_{\bar{Q}_M}(v'-w')=g\cdot f\cdot \tilde{\iota}_{\bar{Q}_{2M}}(m_v-m_w)\in f\cdot \mathbf{R}_d,
\end{equation*}
a contradiction to Lemma \ref{l:nofmultiples}. Hence $\phi$ is injective and 
\begin{displaymath}
\bigl|\pi _{\bar{Q}}({\pi}^{-1}_{\bar{Q}_L \setminus \bar{Q}}(u))\bigl|=\bigl|\pi _{\bar{Q}_L}({\pi}^{-1}_{\bar{Q}_{L+M} \setminus \bar{Q}}(u))\bigl|\le \bigl| \mathcal{W}_{\bar{Q}_L} \bigr|.
\end{displaymath}
\end{proof}
\begin{coro}\label{c:measureentropy}
The topological entropy $\textup{h}_{\textup{top}}(\sigma)$ of $\sigma=\sigma_{\mathcal{W}^{(h)}_g}$ satisfies \[\textup{h}_{\textup{top}}(\sigma) \le \lim_{N \to \infty} \frac{1}{|Q_N|} \log \bigl| \mathcal{W}_{Q_N} \bigr|.\]
\end{coro}
\begin{proof}
From \eqref{eq:Whgsize} and Lemma \ref{l:estimation} we obtain that for any $L > 6M$
\begin{eqnarray*}
\textup{h}_{\textup{top}}(\sigma_{\mathcal{W}^{(h)}_g})&=&\lim_{N\to \infty} \frac{1}{|Q_N|} \log \bigl| \pi_{Q_N}(\mathcal{W}^{(h)}_g) \bigr|
\\
&=&\lim_{N\to\infty }\frac{1}{|Q_N|} \sup_{u \in \pi_{Q_{N+L+M}\setminus Q_N}(\mathcal{W}^{(h)}_g)}\log\,\bigl|\pi _{Q_N}({\pi}^{-1}_{Q_{N+L+M} \setminus Q_N}(u))\bigl| \\
&\le& \lim_{N\to\infty }\frac{1}{|Q_N|} \log \bigl| \mathcal{W}_{Q_{N+L}} \bigr|=\lim_{N\to\infty }\frac{1}{|Q_{N+L}|}\log \bigl| \mathcal{W}_{Q_{N+L}} \bigr| \\
&=&\lim_{N\to\infty }\frac{1}{|Q_{N}|}\log \bigl| \mathcal{W}_{Q_N} \bigr|.
\end{eqnarray*}
\end{proof}
The next statement follows from the fact that $\mu^{(g)}$ is the unique accumulation point (cf. Remark \ref{r:uniqueness}). A proof can be found in \cite[Thm. 6.5]{4Schmidt1}.
\begin{lemm}
The measure $\mu^{(g)}$ is $\sigma_{\mathcal{W}^{(h)}_g}$-invariant.
\end{lemm}
\begin{proof}[Proof of Theorem \ref{t:entropymeasure}]
For $\sigma=\sigma_{\mathcal{W}^{(h)}_g}$ we have that $\textup{h}_\mu^{(g)}(\sigma)\le \textup{h}_{\textup{top}}(\sigma)$, since $\mu^{(g)}$ is concentrated on $\mathcal{W}^{(h)}_g$. For the reverse inequality we set 
\[\nu_L=\frac{1}{|Q_L|} \sum_{\mathbf{k}\in Q_L}\sigma^\mathbf{k}_* \mu^{(g)}_{Q_L}\]
and note that $\nu_L \to \mu^{(g)}$ as $L\to \infty$, since $\mu^{(g)}_{Q_L} \to \mu^{(g)}$ and $\mu^{(g)}$ is $\sigma$-invariant. Let $\Gamma=\max(\gamma_h,\gamma')$, where $\gamma_h$ is defined in Definition \ref{d:dompol} and $\gamma'$ in \eqref{eq:gamma'}. Then the partition $\zeta=\{[j]:0\le j \le \Gamma-1 \}$ is a generator set for $\sigma$, where $[j]=\{ v \in \{0,\dots,\Gamma-1\}^{\mathbb{Z}^d}:v_{\mathbf{0}}=j \}$. For $L \ge 1$ we put $\zeta_L= \bigvee_{\mathbf{j}\in Q_L} \sigma^{-\mathbf{j}} (\zeta)$. Since $\mu^{(g)}_{Q_L}$ is the equidistributed measure on $\mathcal{W}_{Q_L}$ we have that $\log \bigl| \mathcal{W}_{Q_L} \bigr|= H_{\mu_{Q_L}^{(g)}}(\zeta_L)$. Using the same methods as in the proof of \cite[Thm. 5.9]{4Schmidt1} we get for every $M,L \ge 1$ with $2M < L$
\[ 
|Q_M| \log \bigl| \mathcal{W}_{Q_L} \bigr| \le \sum_{\mathbf{i}\in Q_L} H_{\sigma^{\mathbf{i}}_*\mu_{Q_L}^{(g)}}(\zeta_M)+ |Q_M|(|Q_L|-|Q_{L-M}|)\log(\Gamma)
\]
and
\[\frac{|Q_M|}{|Q_L|} \log \bigl| \mathcal{W}_{Q_L} \bigr| \le H_{\nu_L}(\zeta_M)+\frac{|Q_M|(|Q_L|-|Q_{L-M}|)}{|Q_L|}\log(\Gamma).
\]
Applying Corollary \ref{c:measureentropy} when sending $L$ to infinity and dividing subsequently by $|Q_M|$ yields
\[
\textup{h}_{\textup{top}}(\sigma) \le \lim_{L\to \infty} \frac{1}{|Q_L|} \log \bigl| \mathcal{W}_{Q_L} \bigr|\le \lim_{M\to \infty} \frac{1}{|Q_M|}H_{\mu^{(g)}}(\zeta_M)=h_\mu^{(g)}(\sigma).
\]
\end{proof}
We have finally found a measure $\mu^{(g)}$ of maximal entropy on the equal entropy cover $\mathcal{W}^{(h)}_g=\mathcal{R}^{(h)}\cap \ker \xi_g$ of $X_{\tilde{f}}$. As shown in \cite[Thm. 6.6]{4Schmidt1} this measure is unique. We summarise our achievements in the following theorem.
\begin{theo} \label{t:symbolicrepresentationgeneral}
Let $d\ge 1$, and $h=fg \in \mathbf{P}_d$ with $f,g$ satisfying the assumptions in Proposition \ref{p:Delta'}. Then the subshift $\mathcal{W}^{(h)}_g$ (resp. $\mathcal{V}^{(h)}_g$) is a symbolic representation of $X_{\tilde{f}}$ with a unique measure $\mu^{(g)}$ (resp. $\mu^{(f)}=(\xi_g)_*\mu^{(g)}$) of maximal entropy for which the covering map $\xi_h$ (resp. $\xi_f$) restricted to $\mathcal{W}^{(h)}_g$ (resp. $\mathcal{V}^{(h)}_g$) is almost one-to-one. 
\end{theo}
\begin{rema}(1) The statement about $\mathcal{V}^{(h)}_g,\mu^{(f)}$, and $\xi_f$ in the above theorem follows fromt the fact that $\xi_g(\mathcal{W}^{(h)}_g)=\mathcal{V}^{(h)}_g$ (cf. \eqref{eq:Wgh}--\eqref{eq:Vgh}). 

(2) The assumptions on $f$ and $g$ in the above theorem seem artificial but are needed to obtain the Haar measures $\mu^{(g)}_F$ on $\mathcal{W}_F$ in finite volume $F$ which subsequently leads to the unique measure $\mu^{(g)}$ of maximal entropy on $\mathcal{W}^{(h)}_g$. Note that these requirements are not necessary in order to show that $\mathcal{W}^{(h)}_g$ is an equal entropy cover of $X_{\tilde{f}}$ (cf. Theorem \ref{t:equalentropy}). Moreover, for $d>1$ the only non-trivial class of Laurent polynomials which we find to meet these requirements consists of pairs $g,f \in \mathbf{R}_d$ for which $g\in \mathbf{P}_d$ and $f\in \mathbf{P}_d^+$. It would surprise us if this class is indeed the only one. 
\end{rema}

\end{document}